\documentclass[11pt]{amsart}
\usepackage[margin=3.2cm]{geometry}
\usepackage{amssymb,amsthm}
\usepackage{graphicx}
\usepackage[all]{xy}
\usepackage{color} 
\usepackage{enumitem}
\usepackage{tikz}
\usetikzlibrary{patterns} 
\allowdisplaybreaks
\theoremstyle{plain}
\newtheorem{lem}{Lemma}
\newtheorem{prop}{Proposition}
\newtheorem{thm}{Theorem}

\newtheorem*{lem*}{Lemma}
\newtheorem*{prop*}{Proposition}
\newtheorem*{thm*}{Theorem}
\newtheorem*{cor*}{Corollary}
\newtheorem*{conj*}{Conjecture}
 \theoremstyle{remark}

 \newtheorem*{rmk}{Remark}

\newtheorem*{eg}{Example}
\newtheorem*{example}{Example}

\newtheorem*{rmks}{Remarks}
\newtheorem*{Acknowledgements*}{Acknowledgements}
\theoremstyle{definition}

\DeclareMathOperator{\rad}{rad} 

\DeclareMathOperator{\GL}{GL}

\DeclareMathOperator{\Rep}{Rep}

\newcommand{\bil}[2]{\langle #1, #2 \rangle}
\newcommand{\Z}{{\mathbb{Z}}}
\newcommand{\Ext}{\operatorname{Ext}}
\newcommand{\Hom}{\operatorname{Hom}}

\newcommand{\SL}{\mathrm{SL}}

\newcommand{\F}{\mathfrak{F}}
\newcommand{\St}{\mathrm{St}}
\newcommand{\ch}{\operatorname{ch}}
\newcommand{\pr}{\operatorname{pr}}

\newcommand{\R}{\mathbb{R}}
\newcommand{\soc}{\operatorname{soc}}
\newcommand{\Top}{\operatorname{head}}

\renewcommand{\labelitemi}{$\cdot$}
\begin{document}
\title[Decomposition of Tensor Products]{Decomposition of Tensor
  Products of Modular\\Irreducible Representations for $\text{SL}_3$:
  the $p\ge 5$ case} 
\author{C.~Bowman}
%\address{Institut de Math\'ematiques de Jussieu, 175 rue du chevaleret,
%  75013, Paris} 
\address{Department of Mathematics, City University London, 
  Northampton Square, London EC1V 0HB, England}
%\email{Bowman@math.jussieu.fr}
\email{chris.bowman.2@city.ac.uk}
  \author{S.R.~Doty} 
\address{Mathematics and Statistics, Loyola
University Chicago, Chicago, IL 60660, USA}
\email{doty@math.luc.edu} 
\author{S.~Martin}
\address{DPMMS, Centre for Mathematical Sciences, University of Cambridge, 
  Wilberforce Road, Cambridge CB3 0WB, UK}
\email{S.Martin@dpmms.cam.ac.uk}
 \date{\today}
\begin{abstract}
  We study the structure of the indecomposable direct summands of
  tensor products of two \emph{restricted} rational simple modules for
  the algebraic group $\SL_3(K)$, where $K$ is an algebraically closed
  field of characteristic $p \ge 5$.  We also give a
  characteristic-free algorithm for the decomposition of such a tensor
  product into indecomposable direct summands.  The $p<5$ case was
  studied in the authors' earlier paper \cite{BDM1}. We find that for
  characteristics $p\ge 5$ all the indecomposable summands are rigid,
  in contrast to the characteristic 3 case.
\end{abstract}

\maketitle

\section*{Introduction}\noindent
Let $G=\SL_3(K)$ where $K$ is an algebraically closed field of
characteristic $p\ge 5$.  The purpose of this paper, which is a
continuation of \cite{BDM1}, is to describe the family $\F$ of
indecomposable direct summands of a tensor product $L \otimes L'$ of
two simple $G$-modules $L, L'$ of $p$-restricted highest weights.
(All modules considered are rational.) We give a characteristic-free
algorithm for the computation of the decomposition multiplicities of
such a tensor product into indecomposable modules and give structural
information about the indecomposable summands.  Thanks to Steinberg's
tensor product theorem, such information provides a first
approximation toward a description of the indecomposable direct
summands of a general tensor product of two (not necessarily
restricted) simple $G$-modules.  Similar questions were previously
studied for characteristic $p<5$ in \cite{BDM1}, where sharper results
were obtained.  The current paper and its prequel \cite{BDM1} were
motivated by \cite{DH}, which studied the $\SL_2(K)$ case.
 
Our main results are summarized in Theorems \ref{thm:1} and
\ref{thm:2} in Section \ref{sec:thms}. The aforementioned algorithm is
given in \ref{algorithm:1} and \ref{algorithm:2}.  In contrast to what
happens in characteristic $p=3$, in characteristics $p\ge 5$ we find
that all the indecomposable summands are rigid modules (socle series
and radical series coincide). All of the indecomposable summands are
in fact tilting modules, except for certain non-tilting simple modules
and a certain family of non-highest weight modules, which had also
been observed in the $p=3$ case.  The first examples of non-rigid
tilting modules for algebraic groups were exhibited in \cite{BDM1};
further examples and a general positive rigidity result for tilting
modules are now available in \cite{and}.

We had hoped that determining the indecomposable summands of $L
\otimes L'$ in the restricted case would lead to their determination
in general, by some sort of generalised tensor product result; e.g.,
see Lemma 1.1 in \cite{BDM1}. However, our results show this is not
the case, and the general (unrestricted) decomposition problem remains
open. Although our results do in principle give a partial
decomposition in the unrestricted case, using formula (1.1.3) of
\cite{BDM1}, the summands there will not always be indecomposable, and
the problem of finding all splittings of those summands remains
in general unsolved.

\begin{Acknowledgements*}
We thank the referee for his/her careful reading and very helpful comments on this paper.
  The authors also thank S.~Donkin for valuable advice in the preparation
  of this article, and O.~Solberg for advice on using the QPA package
  \cite{QPA}. 
  The first author is grateful for the financial support
  received from grant ANR-10-BLAN-0110.
\end{Acknowledgements*}

\section{Preliminaries}\label{sec:prelim}\noindent
We adopt the notational conventions of \cite{BDM1}. Throughout this
paper $G = \SL_3(K)$ where $K$ is an algebraically closed field of
characteristic $p \ge 5$. This means that $p \ge 2h-2$ where $h=3$ is
the Coxeter number of the underlying root system, so general results
on algebraic groups that are known to hold only for $p \ge 2h-2$ are
available.

\subsection{}\textbf{Weight notations.} \label{ss:8:notation}
Sometimes we need to use both $\GL_3$ and $\SL_3$ weight notation for
calculations.  Any $\GL_3$-module is an $\SL_3$-module by
restriction. A given $\SL_3$-module $M$ may be lifted to a
$\GL_3$-module in many ways, but all such lifts differ by a power of
the determinant representation, and which lift we choose makes no
difference for our results.  It is sometimes convenient to work with
$\GL_3$-modules in order to apply, for example, the
Littlewood--Richardson rule. We adopt the notation of II.1.21 of
\cite{Jantzen}. In particular, $\{ \varepsilon_1, \varepsilon_2,
\varepsilon_3 \}$ is the standard basis of $X(T_{\GL_3}) \simeq \Z^3$
where $T_{\GL_3}$ is the diagonal torus in $\GL_3$. We identify a
$\GL_3$-weight $\chi = a_1 \varepsilon_1 + a_2 \varepsilon_2 + a_3
\varepsilon_3$ with the 3-tuple $((a_1, a_2, a_3)) \in \Z^3$. The
inclusion $\SL_3 \subset \GL_3$ induces an embedding $T_{\SL_3}
\subset T_{\GL_3}$ of their diagonal subgroups, which in turn induces
a surjection $X(T_{\GL_3}) \twoheadrightarrow X(T_{\SL_3})$ given by
restricting characters from $T_{\GL_3}$ to $T_{\SL_3}$, with kernel
generated by $\varepsilon_1+\varepsilon_2+\varepsilon_3$. This map
sends $a_1 \varepsilon_1 + a_2 \varepsilon_2 + a_3 \varepsilon_3$ onto
$(a_1-a_2)\varpi_1 + (a_2-a_3)\varpi_2$, where $\varpi_1, \varpi_2$
are the fundamental weights in $X(T_{\SL_3})$. We shall identify an
$\SL_3$-weight $\lambda_1 \varpi_1 + \lambda_2 \varpi_2$ in
$X(T_{\SL_3}) \simeq \Z^2$ with the ordered pair $(\lambda_1,
\lambda_2)$. In terms of these identifications, the map $X(T_{\GL_3})
\twoheadrightarrow X(T_{\SL_3})$ is given by
\[
  ((a_1,a_2, a_3)) \to (a_1-a_2, a_2-a_3). 
\]
We use double brackets $((\ ,\ ,\ ))$ versus single brackets $(\ ,\ )$
in order to easily distinguish between $\GL_3$ and $\SL_3$-weights.
Given an element $\chi$ in $X(T_{\GL_3})$ we shall denote its image
under the restriction map $X(T_{\GL_3}) \twoheadrightarrow
X(T_{\SL_3})$ by $\overline{\chi}$. In particular, we shall need the
$\SL_3$-weights $\overline{\varepsilon_1} = (1,0)$,
$\overline{\varepsilon_2} = (-1,1)$, and $\overline{\varepsilon_3} =
(0,-1)$ coming from the $\GL_3$-weights $\varepsilon_1=((1,0,0))$,
$\varepsilon_2=((0,1,0))$, and $\varepsilon_3=((0,0,1))$ forming the
standard basis of $X(T_{\GL_3})$. As usual, $\rho = \varpi_1+\varpi_2
= (1,1)$ is one-half the sum of the positive roots for $\SL_3$. 

\subsection{}
Until further notice we use only $\SL_3$ notation for weights. Thus
$X^+ = \{(a,b) \colon a,b \ge 0\}$ the set of dominant weights. The
simple roots are $\alpha_1 = (2,-1)$, $\alpha_2 = (-1,2)$.  As in
\cite{BDM1}, $T(\lambda)$ is the indecomposable tilting module of
highest weight $\lambda$, $\Delta(\lambda)$ the Weyl module of highest
weight $\lambda$, $\nabla(\lambda)$ its contravariant dual, and
$L(\lambda)$ the simple head of $\Delta(\lambda)$. We also denote the
Steinberg module $\Delta(p-1,p-1) = L(p-1,p-1)$ by $\St$.  In case
$\nabla(\lambda) = L(\lambda)$ is simple, we have $T(\lambda) =
\Delta(\lambda) = \nabla(\lambda) = L(\lambda)$; this applies in
particular to the Steinberg module $\St$, the natural module
$E=\Delta(1,0)$ and its linear dual $E^\ast = \Delta(0,1)$. The
tilting modules $T(\lambda)$ are always contravariantly self-dual, and
a tensor product of two tilting modules is again tilting.

By $\mathcal{F}(\Delta)$ we mean the category of $G$-modules admitting
a $\Delta$-filtration, and, dually, by $\mathcal{F}(\nabla)$ we mean
the category of $G$-modules admitting a $\nabla$-filtration.  We note
that $T(\lambda)$ is the unique indecomposable module of highest
weight $\lambda$ in $\mathcal{F}(\Delta) \cap \mathcal{F}(\nabla)$.

\subsection{} 
It is useful to regard a given $G$-module as a module for some
generalised Schur algebra $S(\pi)$, where $\pi$ is an appropriate
finite saturated poset of dominant weights; see \cite{don1} or Chapter
II.A in \cite{Jantzen} for details on generalised Schur algebras.  The
fact that $S(\pi)$ is quasi-hereditary is used repeatedly. For any
$S(\pi)$ we let $P_\pi(\lambda)$ denote the projective cover of
$L(\lambda)$ in the category of $S(\pi)$-modules. We may drop the
subscript $\pi$ in case the set $\pi$ is fixed by the context. The
contravariant dual of $P_\pi(\lambda)$ is isomorphic to the injective
hull of $L(\lambda)$ in the category of $S(\pi)$-modules.

It is known that $P(\mu) \in \mathcal{F}(\Delta)$.  Let $[P(\mu)
  \colon \Delta(\lambda)]$ denote the number of subquotients
isomorphic to $\Delta(\lambda)$ in a $\Delta$-filtration of $P(\mu)$;
this number is known to be independent of the choice of
$\Delta$-filtration. For any finite dimensional $S(\pi)$-module $M$
let $[M \colon L(\mu)]$ be the multiplicity of $L(\mu)$ in a
composition series of $M$.  See Proposition A2.2(iv) of \cite{qschur}
or Theorem 2.6 of \cite{Donkin:filt} for the following basic
reciprocity property, which will be used repeatedly.

\begin{prop}\label{prop:BH-reciprocity}
  Let $S(\pi)$ be the generalised Schur algebra determined by a finite
  saturated set $\pi \subset X^+$. Then $[P(\mu) \colon
    \Delta(\lambda)] = [\nabla(\lambda) \colon L(\mu)]$ for all
  $\lambda, \mu \in \pi$.
\end{prop}

This is sometimes called \emph{Brauer--Humphreys reciprocity}.  The
Schur algebra setting also allows us to make use of the following
refinement of Proposition~\ref{prop:BH-reciprocity} from \cite{BM}. Let
$\rad_i P(\mu)$ be the $i$th radical layer of $P(\mu)$.

\begin{prop}\label{BGG}
  Let $S(\pi)$ be a generalised Schur algebra, where $\pi \subset X^+$
  is a finite saturated set. Then for any $\lambda, \mu \in \pi$ we
  have:
  \begin{align*}
  [ \rad_iP(\mu): L(\lambda)] &=[\rad_iP(\lambda):L(\mu)].
  \intertext{This reciprocity respects the $\Delta$-filtration of the
    projective modules:} [ \rad_iP(\mu): \Top\Delta(\lambda)]&=[
  \rad_i\Delta(\lambda):L(\mu)].
  \end{align*}
\end{prop}

% and $\rad^i P(\mu)$ the $i$th term in the radical series.
By $[\rad_iP(\mu): \Top\Delta(\lambda)]$ we mean the number of
successive quotients $\Delta(\lambda_j)$ in a given \emph{fixed}
$\Delta$-filtration of $P(\mu)$ such that $\lambda_j = \lambda$ and
there is a surjection $\rad^i P(\mu)\to \Delta(\lambda)$ which
carries the subquotient $\Delta(\lambda_j)$ onto $\Delta(\lambda)$.
It is easily checked that this is independent of the choice of
$\Delta$-filtration.  See \ref{BGGeg} for an example.

\begin{rmk}
  The contravariant dual of the above theorem relates the socle layers
  of injective modules, and gives information about where
  $\nabla$-modules occur in a $\nabla$-filtration of an injective
  module. This will also be used where needed.
\end{rmk}

%\subsection{} 
We will need the following basic result from the theory of
quasi-hereditary algebras. This follows for instance from Proposition
A2.2 of \cite{qschur}.

\begin{prop}\label{donkin}  For any $\lambda, \mu \in X^+$, we have:

  {\rm(a)} If $\Ext^1(\Delta(\lambda),\Delta(\mu)) \neq 0$ then $\mu >
  \lambda$.

  {\rm(b)} $\dim_K \Hom_G(M,N) = \sum_{\nu\in X^+} [M: \Delta(\nu)][N:
  \nabla(\nu)]$, for any $M \in \mathcal{F}(\Delta)$, $N \in
  \mathcal{F}(\nabla)$.
\end{prop}

\subsection{}\label{Jantzen-iso}
Let $X_1 = \{(a,b) \colon 0 \le a,b \le p-1 \}$ be the set of
restricted weights. Let $w_0$ be the longest element in the Weyl group
$W$. By Jantzen \cite{jan} we have for any $\lambda \in X_1$ that
\begin{equation}\label{Jantz}
  T(2(p-1)\rho + w_0\lambda) \cong P_\pi(\lambda),
\end{equation}
an isomorphism of $S(\pi)$-modules, where $\pi = \{\lambda \in X^+
\colon \lambda \le 2(p-1)\rho + w_0\lambda) \}$. 

This fact will be used repeatedly in the proof of our results. Any
projective tilting module for $S(\pi)$ is also injective, since
tilting modules are contravariantly self-dual, so the above module is
projective, injective, and tilting, for any $\lambda \in X_1$.

For $p\geq 5$, there is a twisted tensor product theorem for tilting
modules, due to Donkin \cite[(2.1)
  Proposition]{Donkin:Zeit}.   Every $\lambda \in X^+$ has a
unique expression in the form
\begin{equation}  
  \lambda = \textstyle\sum_{j=0}^m a_j(\lambda) \, p^j
\end{equation}
with $a_0(\lambda), \dots, a_{m-1}(\lambda) \in (p-1)\rho + X_1$ and
$\langle a_m(\lambda), \alpha^\vee \rangle < p-1$ for at least one
simple root $\alpha$.  Given $\lambda \in X^+$, express $\lambda$ in
the form above. There is an isomorphism of $G$-modules
\begin{equation}\label{eq:tptilt}
  T(\lambda) \simeq \textstyle \bigotimes_{j=0}^m T(a_j(\lambda))^{[j]}.
\end{equation}

\section{Facets and alcoves for $G =\SL_3$}\label{sec:facets}\noindent
We now introduce a labelling scheme for keeping track of the various
alcoves and facets needed in our calculations.

\subsection{}\label{ss:W}
The Euclidean space associated to the root system of $G$ is $X
\otimes_\Z \R \cong \R^2$. The Weyl group $W = \langle s_1,
s_2\rangle$ is isomorphic to the symmetric group on 3 letters. Here
$s_1, s_2$ are reflections in lines orthogonal to the simple roots
$\alpha_1, \alpha_2$ respectively.

\subsection{}\label{ss:symmetry}
The map $V \mapsto V^*$, where $V^*$ is the linear dual of $V$, is an
involution on the set of $G$-modules. If $V$ is a highest weight
module of highest weight $\lambda$, the highest weight of $V^*$ is
$-w_0(\lambda)$, so this involution on $G$-modules induces a
corresponding involution $\lambda \mapsto -w_0(\lambda)$ on the set
$X^+$, where $w_0 = s_1 s_2 s_1$ is the longest element of $W$.  We
refer to this involution as \emph{symmetry}, and we generally will
omit stating results that can be obtained `by symmetry' from results
already stated.

\subsection{} \label{ss:conventions}   
Let $\rho = \alpha_1 + \alpha_2 = (1,1)$.  Recall that the dot action
of $W_p$ on $X$ is defined by the rule $w \cdot \lambda =
w(\lambda+\rho) - \rho$.
The bottom alcove $C_1$ is defined by
\begin{align*}
  C_1 = \{ v \in \R^2 : 0 < \langle v+\rho, \alpha^\vee \rangle < p
  \text{ for all positive roots } \alpha\}.
\end{align*}
As depicted in Figure \ref{fig:1}, $C_1$ is the interior of an
equilateral triangle in $\R^2$ with one vertex at the point $-\rho$.
The affine Weyl group $W_p$ is generated by the reflections in the
walls of $C_1$.

Any translate $w \cdot C_1$ of $C_1$ under the dot action
of $W_p$ is called an alcove.  The closure of an alcove $C_i$ will be
denoted by $\overline{C}_i$.  In Figure \ref{fig:1} below we number
the alcoves, which we call \emph{fundamental} alcoves, that arise in
our study. Alcoves $i$ and $i'$ are in symmetry according to the
involution of \ref{ss:symmetry}.
\begin{figure}[h]
\begin{center}
\begin{tikzpicture}[scale=4.0]
  % name vertices using polar coordinates
  \path (0,0) coordinate (origin);
  \path (60:7mm) coordinate (A1);
  \path (60:14mm) coordinate (A2);
  \path (60:21mm) coordinate (A3);
  \path (60:28mm) coordinate (A4);

  \path (120:7mm) coordinate (B1);
  \path (120:14mm) coordinate (B2);
  \path (120:21mm) coordinate (B3);
  \path (120:28mm) coordinate (B4);

  \path (A1) ++(120:21mm) coordinate (C1);
  \path (A2) ++(120:14mm) coordinate (C2);
  \path (A3) ++(120:7mm) coordinate (C3);

% define a clip region to crop the resulting figure
  \clip (-14mm,-2mm) rectangle (60:28.1mm);
  % now draw the grid lines
  \foreach \i in {1,...,27}
  {
    \path (origin)++(60:\i mm)  coordinate (a\i);
    \path (origin)++(120:\i mm)  coordinate (b\i);
    \path (a\i)++(120:60mm) coordinate (ca\i);
    \path (b\i)++(60:60mm) coordinate (cb\i);
    \draw[very thin,gray] (a\i) -- (ca\i) (b\i) -- (cb\i);% (a\i)--(b\i); 
  }
  % draw the edges of the alcoves
  \draw[thick] (origin) -- (A4) (origin) -- (B4) (A4) -- (B4)
        (A1) -- (C1) (A2) -- (C2) (A3) -- (C3) 
        (B1) -- (C3) (B2) -- (C2) (B3) -- (C1) 
        (B1) -- (A1) (B2) -- (A2) (B3) -- (A3);
  % now draw the upper bounding lines
  \path (a13)++(120:13mm) coordinate (P);
        % node[align=center, above, fill=white]{\small$(2p-2)\rho$} ;
  \path (60:20.5mm)++(120:-0.2cm) coordinate (QA);
  \path (120:20.5mm)++(60:-0.2cm) coordinate (QB);
  \draw[very thick,gray,densely dashed] (P) -- (QA) (P) -- (QB);
  % and circle the origin
  \draw[fill] (origin)++(60:1mm)++(120:1mm) circle (0.3pt)
         node[align=left, right]{\small$o$};
  \draw[fill] (P) circle (0.3pt)
         node[align=left, above] {};

  % put in the alcove labels
  \draw (origin) ++(90:4.2mm) node {\Large$\mathbf{1}$};
  \draw (origin) ++(90:8.4mm) node {\Large$\mathbf{2}$};
  \draw (A1) ++(90:4.2mm) node {\Large$\mathbf{3}$};
  \draw (A1) ++(90:8.4mm) node {\Large$\mathbf{4}$};
  \draw (B1) ++(90:4.2mm) node {\Large$\mathbf{3}'$};
  \draw (B1) ++(90:8.4mm) node {\Large$\mathbf{4}'$};
  \draw (A2) ++(90:4.2mm) node {\Large$\mathbf{6}$};
  \draw (A2) ++(90:8.4mm) node {\Large$\mathbf{8}$};
  \draw (B2) ++(90:4.2mm) node {\Large$\mathbf{6}'$};
  \draw (B2) ++(90:8.4mm) node {\Large$\mathbf{8}'$};
  \draw (A1) ++(90:16.2mm) node {\Large$\mathbf{9}$};
  \draw (B1) ++(90:16.2mm) node {\Large$\mathbf{9}'$};
  \draw (origin) ++(90:16.2mm) node {\Large$\mathbf{5}$};
  \draw (origin) ++(90:20mm) node {\Large$\mathbf{7}$};
  \draw (origin) ++(90:-0.3mm)  node {$-\rho$};

  % restricted weights
  \draw (origin)++(60:1mm)+(120:1mm) coordinate (AA);
  \draw (A1)++(120:1mm) coordinate (BB);
  \draw (A1)++(120:7mm) coordinate (CC);
  \draw (B1)++(60:1mm) coordinate (DD);
  \filldraw[green,nearly transparent] (AA) -- (BB) -- (CC) -- (DD) -- (AA);

%  \begin{scope}[arrow]
%  \path[draw,shorten >=0.1cm] (A3)++(120:8mm) node[anchor=west]
%       {\small$(2p-2)\rho$} to[out=150, in=45] (P);
%  \end{scope}
\end{tikzpicture}
\end{center}
\caption{Labelled alcoves for $p=7$.  The southernmost vertex is
  $-\rho$. The origin $o = (0,0)$ is the marked point directly above
  it in alcove 1. The unlabelled marked point in alcove 7 is
  $2(p-1)\rho$. The region of dominant weights $\lambda$ with $\lambda
  \le 2(p-1)\rho$ is bounded above by the two dashed lines emanating
  from that point.  The labelled alcoves, and walls between them,
  contain all the points in the region. The coloured region of points
  is the restricted region $X_1$.  The simple root $\alpha_1$
  (respectively $\alpha_2$) is orthogonal to the wall of $C_1$
  pointing northeast (respectively, northwest) from $-\rho$. }
\label{fig:1}
\end{figure}
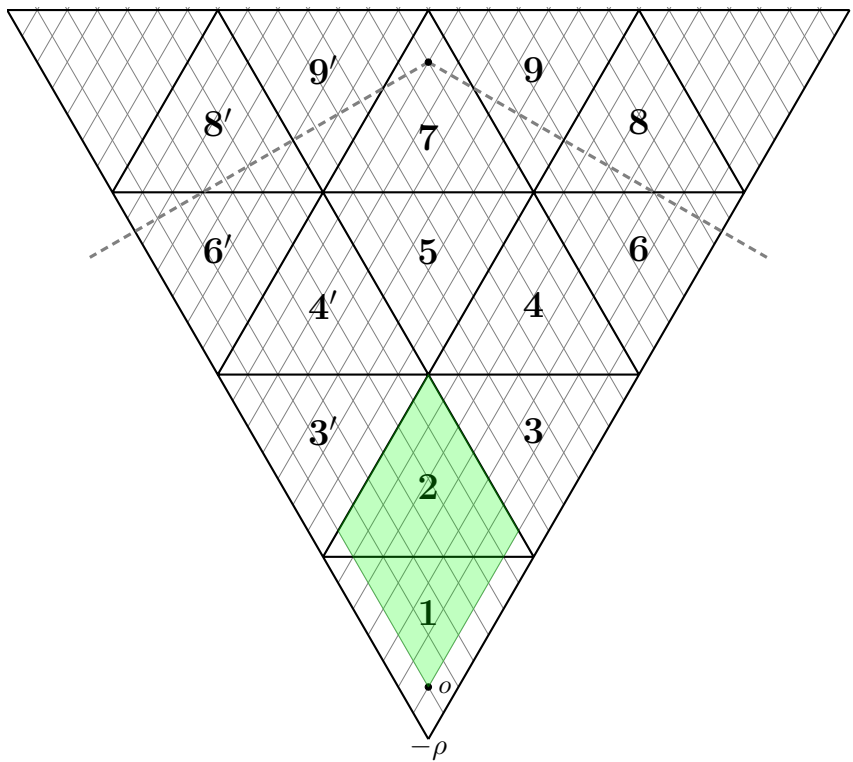
We let $\mathcal{F}_{i|j}= \overline{C}_i\cap \overline{C}_j$
denote the wall between any pair $C_i, C_j$ of adjacent alcoves; this
wall is a facet in the sense of \cite[II.6.2]{Jantzen}. The reflection
in the wall $\mathcal{F}_{i|j}$ will be denoted by $s_{i|j}$. In
this notation the generators of $W_p$ are $s_{1|2}$ along with the
elements $s_1$, $s_2$ defined in \ref{ss:W}.

\subsection{}
The points at intersections of very light grid lines in Figure
\ref{fig:1} are weights. The region of highest weights of composition
factors that can occur in a $p$-restricted tensor product $L \otimes
L'$ is the set of all dominant weights $\lambda$ such that $\lambda
\le (2p-2)\rho$; this is the set of weights on or below the dashed
lines in Figure \ref{fig:1}. The alcoves in question are the numbered
ones in the figure.

In case $p<7$, not all of the labelled alcoves in Figure \ref{fig:1}
actually appear in restricted tensor product decompositions.  The
degeneracies for $p<7$ are caused by fewer points appearing in each
alcove; to understand this the reader is encouraged to draw the
analogue of Figure \ref{fig:1} for the smaller primes, after which the
degeneracies are apparent.

\section{Main Results}\label{sec:thms}\noindent
In this section we give the two main results of the paper.  The first
main result describes the members of the family $\F= \F(G)$ of
indecomposable direct summands of a tensor product of two restricted
simple $G$-modules. The second main result is a description of the
structure of the modules in $\F$, as far as we can deduce structural
information by current methodology. The highest weight modules in $\F$
are either simple modules $L(\lambda)$ or indecomposable tilting
modules $T(\lambda)$ for various $\lambda$, however there are non
highest weight modules $M(\lambda)$ which occur in $\F$.

\subsection{Notational conventions.}\label{ss:notcon}
We assume the reader is familiar with Jantzen's translation principle
\cite{Jantzen}, which in particular implies an equivalence of module
structure for highest weight modules of the form $L(\lambda)$,
$\Delta(\lambda)$, $\nabla(\lambda)$, or $T(\lambda)$ belonging to the
same facet. We will therefore adopt facet notation for highest weight
modules whenever convenient.  This replaces a highest weight
$\lambda \in X^+$ by its corresponding alcove label $j$ whenever
$\lambda \in C_j$. Thus, for example, given $\lambda \in C_1 \cap
X$, $L(\lambda)$ is denoted by $L(1)$, $\Delta(\lambda)$ is denoted by
$\Delta(1)$, $T(\lambda)$ is denoted by $T(1)$, and so on.
Furthermore, within module diagrams, for each label $j$ the simple
module $L(j)$ will be identified with its alcove label $j$.  For
$p$-singular weights $\lambda \in \mathcal{F}_{i|j}$ lying on the wall
common to two alcoves $i$ and $j$ (but not a vertex of any alcove) we
use the notation $i|j$ to denote the facet, and use the notation
$L(i|j)$, $\Delta(i|j)$, $T(i|j)$ for the highest weight modules
$L(\lambda)$, $\Delta(\lambda)$, $T(\lambda)$ respectively. Similar to
the above, within module diagrams, for each facet label $i|j$ the
simple module $L(i|j)$ will be identified with its facet label $i|j$.

\subsection{}\label{ss:M}
Corresponding to each weight $\lambda \in C_2$, there is a unique
indecomposable module $M(\lambda)$ in $\mathfrak{F}$ which is
\emph{not} generated by a single highest weight vector. For $\lambda
\in C_2$, the module $M(\lambda)$ has a simple socle and head
isomorphic to $L(\lambda)$, with the quotient $\rad M(\lambda)/ \soc
M(\lambda)$ of the radical by the socle a semisimple module isomorphic
to $L(s_{2|3'} \cdot \lambda) \oplus L(s_{2|3} \cdot \lambda) \oplus
L(s_{1|2} \cdot \lambda)$. The module $M(\lambda)$ may be constructed
as a submodule of the tilting module $T(s_{3|4} \cdot s_{2|3} \cdot
\lambda)$ or, symmetrically, as a submodule of the tilting module
$T(s_{3'|4'} \cdot s_{2|3'} \cdot \lambda)$. We construct it in
section \ref{M} as a quotient of an appropriate generalised Schur
algebra.  Similar modules (with identical structure diagrams) already
appeared in \cite{BDM1} for $p=3$.

Making use of our convention (from \ref{ss:notcon}) of replacing
highest weights by their alcove or facet labels we often write $M(2)$
instead of $M(\lambda)$ for $\lambda \in C_2$.  In this notation,
$M(2)$ is isomorphic to a submodule of $T(4)$ or $T(4')$. The Alperin
diagram of $M(2)$ is
\[ \scriptstyle M(2) = 
\begin{minipage}{30mm}
\def\objectstyle{\scriptstyle}\xymatrix@=6pt{
&		&2\ar@{-}[dr]\ar@{-}[dl]\ar@{-}[d]		&\\
&3\ar@{-}[dr]		&1\ar@{-}[d]		&3'\ar@{-}[dl]\\
&		&2		&		&			\\	
}
\end{minipage}
\]
and this is a strong Alperin diagram. (Recall from \cite{BDM1} that a
diagram is said to be strong if it determines the socle as well as
radical series.)  Although $M(2)$ is not a highest weight module, it
still has simple head isomorphic to $L(2)$, so this notation should
not cause confusion. The module $M(2)$ is rigid and self-dual under
contravariant duality.

Let $\F=\F(\SL_3)$ be the family of isomorphism classes of
indecomposable direct summands of a $p$-restricted tensor product $L
\otimes L'$, where $L, L'$ are $p$-restricted simple modules. We
studied this family in \cite{BDM1} for $p=2,3$. For $p=2$ we found
\cite[Prop.~3.2]{BDM1} that $\F$ is precisely the set
\begin{equation*}
\F_R = \{ T(\lambda): 0 \le \bil{\lambda+\rho}{\alpha^\vee} \le 2p-2,
\text{ for all simple roots $\alpha$} \}.
\end{equation*}
For $p=3$, we found \cite[Prop.~4.2]{BDM1} that $\F_R \subset \F$, but
there are also a few other modules in $\F$, which we called
`exceptional'. It turns out that $\F_R$ is contained in $\F$ for every
characteristic $p$. We call the members of $\F_R$ \emph{regular} and
all other members of $\F$ \emph{exceptional}.  In the following, we
have incorporated part of the results of \cite{BDM1} in order to
summarize our results for all $p$.

\renewcommand{\thethm}{\Alph{thm}}
\begin{thm} \label{thm:1} 
Let the characteristic $p>0$ of $K$ be arbitrary. Let $\F=\F(\SL_3)$
be the family of isomorphism classes of indecomposable direct summands
of a $p$-restricted tensor product $L \otimes L'$, where $L, L'$ are
$p$-restricted simple modules. Then $\F$ consists of the following
indecomposable modules:
\begin{enumerate}[label={\rm(\alph*)},leftmargin=*,itemsep=0.5em]

\item The set $\F_R$ of regular tilting modules.

\item The exceptional modules $L(\lambda)$ and $M(\lambda)$, for each
  $\lambda \in C_2$.

\item A finite list, depending on $p$, of exceptional tilting
  modules of the form $T(\lambda)$, for various $\lambda$ not already
  listed in part (a). For $p=2$ there are no exceptional tilting
  modules; for $p=3$ there are precisely four (see \cite{BDM1}) of
  highest weight lying on the boundary of $\overline{C}_6$; for
  larger $p$ the number of exceptional modules grows with $p$ with the
  highest weight of such modules lying in the region $\overline{C}_6
  \cup \overline{C}_8 \cup \overline{C}_9$ (and those obtained by
  symmetry).
\end{enumerate}
\end{thm}

\begin{rmks}
1. One can explicitly determine the decomposition of a tensor
product of two $p$-restricted irreducible modules, by an algorithm
described in Section \ref{decomp}.

2.  We include results of \cite{BDM1} for $p=2, 3$ in the theorem, for
completeness. In case $p=2$ the alcove $C_2$ is empty, so part (b) of
the theorem is vacuous, and thus for $p=2$ the members of $\F$ are
just the tilting modules listed in part (a). 
\end{rmks}

For $p \ge 3$ there are three vertices (points common to the closure
of six alcoves) in the admissible region of weights defined in Figure
\ref{fig:1}, and each of them gives the highest weight of a (simple)
tilting member of $\F$. These are the Steinberg module $\St =
T((p-1)\rho) = L((p-1)\rho)$ and the two modules $T(p\varpi_1 +
(p-1)\rho) = L(p\varpi_1 + (p-1)\rho) \cong E^{[1]} \otimes \St$,
$T(p\varpi_2 + (p-1)\rho) = L(p\varpi_2 + (p-1)\rho) \cong (E^*)^{[1]}
\otimes \St$.

Our second main result describes the structure of the other tilting
members of $\F$, of highest weight lying in alcoves or on walls
between a pair of alcoves, assuming $p \ge 5$.

\begin{thm}   \label{thm:2}
Assume that the characteristic of $K$ is $p \ge 5$.
\begin{list}{\labelitemi}{\leftmargin=1em}
\item[(a)] The $p$-singular tilting modules in $\F=\F(\SL_3)$ of
  highest weight lying on walls are all rigid, with structure as
  follows.  The uniserial modules of highest weight lying on walls
  are:
\begin{align*}
T(1|2)=[(1|2)] ; \ \ & \ \
 T(2|3)=[(2|3)]; \\
T(3|4)=[(2|3'),(3|4),(2|3')]; \ \ & \ \
T(4|6)=[(1|2),(4|6),(1|2)]
\end{align*}
along with their symmetric versions. 

The non-uniserial modules of highest weight lying on walls are
$T(4|5)$, $T(8|9)$, $T(6|8)$, $T(5|7)$ and $T(7|9)$, with
structure given by the following strong Alperin diagrams,
respectively:
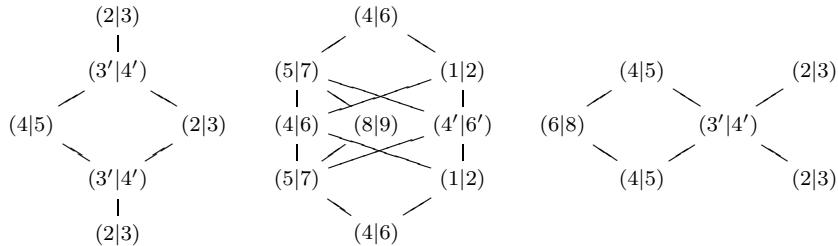
\begin{align*}
 \begin{minipage}{60mm}
\def\objectstyle{\scriptstyle}
\xymatrix@=6pt{
					&(2|3)	\ar@{-}[d]	&	\\
					&(3'|4')	\ar@{-}[dr]\ar@{-}[dl]	&	\\
(4|5)			&			&(2|3)		\\
					&(3'|4')	\ar@{-}[ul]\ar@{-}[ur]	\ar@{-}[d]	&		\\
					&(2|3)			&		  }
 \end{minipage} \quad
  \begin{minipage}{60mm}
\def\objectstyle{\scriptstyle}
\xymatrix@=6pt{
						&(4|6)	\ar@{-}[dr]\ar@{-}[dl]		&	\\
(5|7)	\ar@{-}[drr]\ar@{-}[dr]\ar@{-}[d]	&					&(1|2)\ar@{-}[d]\ar@{-}[dll]\\
(4|6)					&(8|9)				&(4'|6')		\\
(5|7)	\ar@{-}[urr]\ar@{-}[ur]\ar@{-}[u]	&					&(1|2)\ar@{-}[u]\ar@{-}[ull]\\
			&(4|6)	\ar@{-}[ur]\ar@{-}[ul]			&		  }
 \end{minipage} \quad
  \begin{minipage}{60mm}
\def\objectstyle{\scriptstyle}
\xymatrix@=6pt{
						& (4|5)\ar@{-}[dr] 			& 								&(2|3)					\\
(6|8) \ar@{-}[dr] \ar@{-}[ur] & 						& (3'|4')  \ar@{-}[dr] \ar@{-}[ur] 					&		 			\\
 			& (4|5)  	\ar@{-}[ur] 				& 								& (2|3)   
  }
 \end{minipage}
\end{align*}

\begin{align*}
%T(4|5)
\def\objectstyle{\scriptstyle}
\xymatrix@=6pt{
		&	 								& (1|2)	  \ar@{-}[dl] 	 \ar@{-}[dr] 	& 					&\\
		&(4'|6') \ar@{-}[dl] 	 \ar@{-}[dr] 			&		    					& (4|6) \ar@{-}[dr] \ar@{-}[dl] 		& \\
		(1|2)\ar@{-}[dr] 	\ar@{-}[drrr] &									& (5|7)  \ar@{-}[dr] \ar@{-}[dl]	& 	& (1|2) \ar@{-}[dl] \ar@{-}[dlll] 	\\
		&(4'|6')							& 				  			&(4|6) \\
		&									& ( 1,2) \ar@{-}[ur] \ar@{-}[ul] 		&
}\quad
%T(7|9)
\def\objectstyle{\scriptstyle}
\xymatrix@=6pt{
		&	 								& (4|5)	  \ar@{-}[dl] 	 \ar@{-}[dr] 	& 					&\\
		&(3'|4') \ar@{-}[dl] 	 \ar@{-}[dr] 			&		    					& (6|8) \ar@{-}[dr] \ar@{-}[dl] 		& \\
		(2|3)\ar@{-}[dr] 	&									& (7|9)  \ar@{-}[dr] \ar@{-}[dl]	& 	& (4|5) \ar@{-}[dl] 	\\
		&(3'|4')							& 				  			&(6|8) \\
		&									& ( 4,5) \ar@{-}[ur] \ar@{-}[ul] 		&
}
%\end{minipage}
\end{align*}
along with their symmetric versions.

\item[(b)] The $p$-regular tilting modules in $\F$ are all rigid. 
The uniserial ones have the following structure:
$
T(1)=[1]; \quad T(2)=[1,2,1].
$

The non-uniserial ones for which the structure can be completely
worked out are $T(3)$, $T(4)$, and $T(6)$ with the following strong
Alperin diagrams, respectively:
\[
%T(3)  
\begin{minipage}{34mm}
\def\objectstyle{\scriptstyle}\xymatrix@=6pt{
&	&2  \ar@{-}[dl]   \ar@{-}[dr]	&	\\
&3	&	&1	\\
&	&2  \ar@{-}[ul]  \ar@{-}[ur]	}
 \end{minipage}\quad 
% \begin{minipage}{34mm}
%  \def\objectstyle{\scriptstyle}\xymatrix@=6pt{
%&&	 					& 2 \ar@{-}[dr] \ar@{-}[dl] 			& 			&\\
%&&3\ar@{-}[dl] 	& 1\ar@{-}[u] \ar@{-}[d]  \ar@{-}[dll] 	\ar@{-}[drr] 					&3'  \ar@{-}[dr]	& \\
%& 2&			& 4  \ar@{-}[dr] \ar@{-}[dl] \ar@{-}[ur] \ar@{-}[ul] \ar@{-}[u]		& 	& 2\\
%&&3'  \ar@{-}[ul]	& 1\ar@{-}[u]  \ar@{-}[d]  \ar@{-}[urr]  \ar@{-}[ull]  			&3\ar@{-}[ur] \\
%&&			& 2 \ar@{-}[ur] \ar@{-}[ul] 		&\\
% } \end{minipage}\quad  
%\begin{minipage}{34mm}
%  \def\objectstyle{\scriptstyle}\xymatrix@=6pt{
%&&	 					& 2 \ar@{-}[dr] \ar@{-}[dl] 			& 			&\\
%&&3\ar@{-}[dl] 	& 1\ar@{-}[u] \ar@{-}[d]  \ar@{-}[dll] 	 					&3'  \ar@{-}[dlll] \ar@{-}[dr]	& \\
%& 2&			& 4  \ar@{-}[dr] \ar@{-}[dl] \ar@{-}[ur] \ar@{-}[ul] \ar@{-}[u]		& 	& 2\\
%&&3'  \ar@{-}[ul] \ar@{-}[urrr]  			& 1\ar@{-}[u]  \ar@{-}[d]  \ar@{-}[urr]  	&3\ar@{-}[ur] \\
%&&			& 2 \ar@{-}[ur] \ar@{-}[ul] 		&\\
% } \end{minipage}
% \quad
  \begin{minipage}{34mm}
  \def\objectstyle{\scriptstyle}\xymatrix@=6pt{
&&	 					& 2 \ar@{-}[dr] \ar@{-}[dl] 			& 			&\\
&&3\ar@{-}[dl]   	& 1\ar@{-}[u] \ar@{-}[d]  \ar@{-}[dll] 	  \ar@{-}[drr] 	 					&3'   \ar@{-}[dr]	& \\
& 2  &			& 4  \ar@{-}[dr] \ar@{-}[dl] \ar@{-}[ur] \ar@{-}[ul] \ar@{-}[u]		& 	& 2\\
&&3'  \ar@{-}[ul]  			& 1\ar@{-}[ull]  \ar@{-}[u]  \ar@{-}[d]  \ar@{-}[urr]  	&3\ar@{-}[ur] \\
&&			& 2 \ar@{-}[ur] \ar@{-}[ul] 		&\\
 } \end{minipage}  \quad 
 \begin{minipage}{34mm}
\def\objectstyle{\scriptstyle}\xymatrix@=6pt{
&				& 											&1 	\ar@{-}[dl]\ar@{-}[dr]				 	  \\
  &				& 4\ar@{-}[dr] 	\ar@{-}[dr]\ar@{-}[d]\ar@{-}[dl]\ar@{-}[drr]		&   									&2 \ar@{-}[dll]\ar@{-}[d]\ar@{-}[dr]		  \\
&1\ar@{-}[dr]  	&3 \ar@{-}[d]								&6 \ar@{-}[dr]					&3'\ar@{-}[d]						&1\ar@{-}[dl] 	 	 \\
  &		&2 \ar@{-}[dr]\ar@{-}[urr]						& 								&4	\ar@{-}[ull]\ar@{-}[dl]			 \\
   &		& 								&1 								& 					 
}
 \end{minipage}
\]
along with their symmetric versions.  In the larger cases we give only
the Loewy structure of the tilting modules.  We highlight the Weyl
filtrations below for $T(5)$ and $T(7)$, respectively:
\begin{align*} 
&\begin{minipage}{60mm}
%T(5)
\def\objectstyle{\scriptstyle}\xymatrix@=6pt{	
&		&						&							&2\ar@{-}[d]	\\
&		&3\ar@{-}[dl]				&							&1				&								&3'\ar@{-}[dr]		&	\\
&2								&							&4\ar@{-}[d]\ar@{-}[dl]\ar@{-}[dll]						&				&4'\ar@{-}[d]\ar@{-}[dr]\ar@{-}[drr]		&		&2	\\
&3		&1						&3'							&5\ar@{-}[dr]\ar@{-}[dl]	&3								&1		&3'		\\
&2\ar@{-}[u]\ar@{-}[urr]\ar@{-}[ur]	&	&4'	\ar@{-}[dr]\ar@{-}[dl]						&				&4	\ar@{-}[dr]\ar@{-}[dl]		&		&2\ar@{-}[u]\ar@{-}[ul]\ar@{-}[ull]		&		\\
&		&3	\ar@{-}[drr]\ar@{-}[urrr]&							&1\ar@{-}[d]		&								&3'\ar@{-}[dll]\ar@{-}[ulll]		&		\\		
&		&						&							&2				&								&		&\\
}\end{minipage}
%\begin{align*}
&\begin{minipage}{90mm}
%T(7)
\def\objectstyle{\scriptstyle}
\xymatrix@=6pt{	
&			&			&							&					&			&1						&						&							&			&			\\
&			&			&4\ar@{-}[d]\ar@{-}[dl]\ar@{-}[dr]	&					&			&2\ar@{-}[dl]			&					&						&4'\ar@{-}[dr]\ar@{-}[d]\ar@{-}[dl]\ar@{-}[dr] 	&			&			\\
&6\ar@{-}[d]	&3	&1				&3'		&1	&						&5\ar@{-}[d]\ar@{-}[dr]&3			&1						&3'		&6'	\ar@{-}[d]		\\
%&\\
&4\ar@{-}[d]&	2\ar@{-}[u]\ar@{-}[urr]\ar@{-}[ur]	&&7\ar@{-}[dr]\ar@{-}[d]\ar@{-}[drr]\ar@{-}[dl]\ar@{-}[dll] &			&&4' \ar@{-}[dr]\ar@{-}[d]\ar@{-}[drr]&					4\ar@{-}[dr]\ar@{-}[d]\ar@{-}[dl]&		&2\ar@{-}[u]\ar@{-}[ul]\ar@{-}[ull]		&			4'\ar@{-}[d]\\
%&\\
&1			&6			&3			&5			&3'		&6'		&3		&1		&3'		&		&1		&\\
& 			&			&4\ar@{-}[u]\ar@{-}[urr]\ar@{-}[ur]\ar@{-}[ul]			&			&4'	\ar@{-}[u]\ar@{-}[ul]\ar@{-}[ur]	&		&		&2	\ar@{-}[u]\ar@{-}[ul]\ar@{-}[ur]	&		&		&		&\\
& 			&			&			&1	 \ar@{-}[ul]\ar@{-}[ur]		&		&		&		&		&		&		&		&\\
&\\
}
\end{minipage} 
\end{align*}
and below we provide the Weyl filtrations of $T(9)$ and $T(8)$, respectively:
\begin{align*}
&\begin{minipage}{80mm}
%T(9)
\def\objectstyle{\scriptstyle}\xymatrix@=6pt{	
&&&		&						&							&4\ar@{-}[d]\ar@{-}[dr]\ar@{-}[dl]	\\
&&&6\ar@{-}[dr]		& 				&3							&1				&3'								&	&5	\ar@{-}[dr]\ar@{-}[drrr]	&	\\
&&7	\ar@{-}[dll]\ar@{-}[d]\ar@{-}[dl]\ar@{-}[dr]\ar@{-}[drr]	&						&4\ar@{-}[dr]							&				&2\ar@{-}[ur]\ar@{-}[ul]\ar@{-}[u]				&&8\ar@{-}[dl]\ar@{-}[drl]\ar@{-}[drlr]				&&4\ar@{-}[d]\ar@{-}[drr]&&4'\ar@{-}[d]\ar@{-}[dl]\ar@{-}[dll]	\\
6&3&5		&3'						&6'	&1						&9\ar@{-}[dr]\ar@{-}[dl]	&5								&3		&6	&3&1\ar@{-}[ul]	&3'	\\
&4'\ar@{-}[ul]\ar@{-}[ur]\ar@{-}[urr]\ar@{-}[u]&	&4\ar@{-}[ul]\ar@{-}[ur]\ar@{-}[ull]\ar@{-}[u]	&	&8	\ar@{-}[dl]\ar@{-}[drr]\ar@{-}[d]						&				&7\ar@{-}[d]	\ar@{-}[dr]\ar@{-}[dll]		&		&4\ar@{-}[u]\ar@{-}[ul]\ar@{-}[ull]		&&2\ar@{-}[u]\ar@{-}[ul]\ar@{-}[ur]		\\
&&1\ar@{-}[ul]\ar@{-}[ur]	&	&3	\ar@{-}[drr]\ar@{-}[urrr]&					5\ar@{-}[dr]	&	&	6	\ar@{-}[dl]						&3'\ar@{-}[dll] 		&		\\		
&&&		&						&							&4				&								&		&\\
}\end{minipage}
\begin{minipage}{60mm}
%T(8)
\def\objectstyle{\scriptstyle}
\xymatrix@=6pt{
& 		& 		& 2 \ar@{-}[dl]		&		 &&&4  \ar@{-}[dl]\ar@{-}[dr]\ar@{-}[d]		\\
&&1		&3		&6		&5			&3			&1		&3'			\\
&&2\ar@{-}[ur]		&4\ar@{-}[dl]\ar@{-}[ur]				&8			&4\ar@{-}[u]		&		&4'\ar@{-}[ull]		&2\ar@{-}[ull]\ar@{-}[ul]\ar@{-}[u]		&	\\
&&1				&3\ar@{-}[ur]	&5\ar@{-}[u]		&6\ar@{-}[ul]				&3\ar@{-}[ul]\ar@{-}[ur]		&1\ar@{-}[ull]\ar@{-}[u]		&3'\ar@{-}[ul]\ar@{-}[ulll]		&	\\
&&&		&4\ar@{-}[u]\ar@{-}[ur]\ar@{-}[ul]		&		&	&2\ar@{-}[ul]	\ar@{-}[ur]\ar@{-}[u]	&		&	\\
&\\  
}
\end{minipage}
\end{align*}

\noindent
where the symmetric versions of these modules are not listed, as usual.  
\end{list}
\end{thm}

Note that all members of $\F$ have a simple $p$-restricted
$G_1T$-socle (and head) except for $T(8)$ and $T(6|8)$ for $p \ge 5$.
Therefore every direct summand in the decomposition (1.1.3) of
\cite{BDM1} is indecomposable, unless it involves a factor of the form
$T(\lambda)$, for $\lambda \in C_8 \cup \mathcal{F}_{6|8}$.  Because
of the upper bound constraint of $2(p-1)\rho$ on the highest weights
of tilting members of $\F$, as depicted in Figure \ref{fig:3}, when
$p=5$ there are no members $T(\lambda)$ in $\F$ with $\lambda \in C_8
\cup C_9$, although such modules do appear for $p > 5$.

The rest of the paper is devoted to the proof of these results.  The
proof of Theorem \ref{thm:2} is given in Sections \ref{tilting} and
\ref{sec:p-reg}. The proof of Theorem \ref{thm:1} is given in Sections
\ref{decomp} and \ref{sec:c2c2}. First we need the structure of
certain Weyl modules, which are summarized in the next section.

\section{The Structure of Certain Weyl Modules}\label{sec:Weyl}
\noindent
In \cite{parker} the $p$-filtration structure of Weyl modules for
$\SL_3(K)$ is determined for all primes.  When these layers are
semisimple this gives the radical structure of the Weyl modules.
Therefore when $p\ge5$ and we consider weights from the first
$p^2$-alcove this re-derives the generic structures calculated in
\cite{DS1} and \cite{ron}, which we shall recall below.  The
$p$-filtrations are semisimple for all but three of the Weyl modules
considered in \cite{BDM1}.  For a given prime these calculations can
also be checked using the Weyl module GAP package available on the
second author's web page.

We remind the reader of the notational conventions of
\ref{ss:conventions} and in particular that in diagrams we will
identify simple modules with their facet label.  The structure of the
$p$-singular Weyl modules in question is given by the following strong
Alperin diagrams, where as in \cite{DH}, \cite{BDM1} we use the
notation $[L_1, L_2, \dots, L_s]$ to depict the structure of the
unique uniserial module $M$ with composition factors $L_1, \dots, L_s$
arranged so that $\rad_iM \cong L_i$ for all $i$.  \begin{align*}
  \Delta(1|2)=[(1|2)], \ \ & \ \ \Delta(2|3)=[(2|3)],
  \\ \Delta(3|4)=[(3|4),(2|3')], \ \ & \ \ \Delta(4|6)=[(4|6),(1|2)],
  \\ \Delta(4|5)=[(4|5),(3'|4'),(2|3)],\ \ &
  \ \ \Delta(6|8)=[(6|8),(4|5)],\\ \Delta(8|9)=[(8|9),(5|7),(4|6)],\ \ &
  \\ \Delta(5|7)= \begin{minipage}{34mm}
    \def\objectstyle{\scriptstyle} \xymatrix@=6pt{ & (5|7) \ar@{-}[dr]
      \ar@{-}[dl] & \\ (4'|6') & & (4|6) \\ & (1|2) \ar@{-}[ur]
      \ar@{-}[ul] & }\end{minipage},\ \ &
  \ \ \Delta(7|9)= \begin{minipage}{34mm}
    \def\objectstyle{\scriptstyle} \xymatrix@=6pt{ & (7|9) \ar@{-}[dr]
      \ar@{-}[dl] & \\ (6|8) & & (3'|4') \\ & (4|5) \ar@{-}[ur]
      \ar@{-}[ul] & }\end{minipage}.
\end{align*}
The structure of the $p$-regular Weyl modules we need is as follows,
where once again each diagram is a strong Alperin diagram.
\begin{align*}
\Delta(1)=[1], \ \ \ \Delta(2)=[2,1], \ \ \ \  \ \Delta(3)=[3,2], \ \ \ \  \ \Delta(6)=[6,4,1],
\end{align*}
\begin{align*}\Delta(4) =\begin{minipage}{34mm}
\def\objectstyle{\scriptstyle}
\xymatrix@=6pt{
  &4 \ar@{-}[dl] \ar@{-}[dr] & \\
3\ar@{-}[dr] &1 \ar@{-}[d] \ar@{-}[u] &3'\ar@{-}[dl] \\
  & 2&
} \; \end{minipage}, \ \ \  \ 
\Delta(8) =\begin{minipage}{34mm}
\def\objectstyle{\scriptstyle}
\xymatrix@=6pt{
  &8\ar@{-}[dl] \ar@{-}[dr] & \\
3\ar@{-}[dr] &5 \ar@{-}[d] \ar@{-}[u] &6\ar@{-}[dl] \\
  & 4&
} \; \end{minipage},
\end{align*}
\begin{align*}
\Delta(5) =\begin{minipage}{34mm}
\def\objectstyle{\scriptstyle}
\xymatrix@=6pt{
&5\ar@{-}[dl] \ar@{-}[dr] \\
4\ar@{-}[drr] \ar@{-}[d] \ar@{-}[dr] &		&4' \ar@{-}[d] \ar@{-}[dl] \ar@{-}[dll] \\
3\ar@{-}[dr] &1 \ar@{-}[d] \ar@{-}[ul] &3'\ar@{-}[dl] \\
  & 2&
}  \end{minipage}, \ \ \ \ 
\Delta(7) =\begin{minipage}{34mm}
\def\objectstyle{\scriptstyle}
\xymatrix@=6pt{
&		&		&7\ar@{-}[dl] \ar@{-}[dr]\ar@{-}[dll] \ar@{-}[drr]\ar@{-}[d]		&			&\\
&6		&3		&5		&3'			&6'\\
&		&4\ar@{-}[ul] \ar@{-}[ur]	\ar@{-}[urr] \ar@{-}[u]			&		&4'	\ar@{-}[ul] \ar@{-}[ur]	\ar@{-}[ull] \ar@{-}[u]			&\\
&		&		&1\ar@{-}[ul] \ar@{-}[ur]		&			&\\
}  \end{minipage},  \ \ \ \ 
\Delta(9) =\begin{minipage}{34mm}
\def\objectstyle{\scriptstyle}
\xymatrix@=6pt{
			&			&9	\ar@{-}[dl] \ar@{-}[dr]\\
			&8	\ar@{-}[dl] \ar@{-}[drr]\ar@{-}[d]			&			&7\ar@{-}[d] \ar@{-}[dr]\ar@{-}[dll] \ar@{-}[dlll]&\\
3			&5				&			&6		&3'\\
			&				&4	\ar@{-}[ul] \ar@{-}[ur]\ar@{-}[ull] \ar@{-}[urr]
}  \end{minipage}.
\end{align*}
All of these Weyl modules, including the $p$-singular ones, are rigid.

\section{The $p$-singular tilting modules}\label{tilting}\noindent
We now begin the proof of Theorem \ref{thm:2}. The characters of the
tilting modules for $G = \SL_3$ are known \cite{Jensen,Parker} (see
also \cite{deV}) so our task is just to prove the structural results
in Theorem \ref{thm:2}. This proof is split over the next two
sections. The present section considers only the $p$-singular case
while the next considers the $p$-regular case.

\subsection{} 
It turns out that most tilting modules $T(\lambda)$ that we consider
are projective for the generalised Schur algebra $S(\le \lambda)$.  To
prove injectivity (and hence projectivity) one need only check that
\begin{enumerate}
  \item[(a)] the socle of the tilting module is simple, and

  \item[(b)] the character of the relevant projective module (given by
    Proposition \ref{prop:BH-reciprocity}) coincides with the tilting
    character.
\end{enumerate}
Note that part (a) can usually be done by constructing an injection
into a tilting module of the form $T(2(p-1)\rho+\varpi_0\lambda)$ for
$\lambda\in X_1$ and applying \ref{Jantzen-iso}.

We begin with the $p$-singular tilting modules, not only because their
structures tend to be less complicated, but also because their images
under translation functors provide useful filtrations of the
$p$-regular tilting modules considered in the next section.
 
\subsection{}\label{BGGeg}
We shall build the tilting modules $T(2|3), T(3'|4')$ and $T(4|5)$ as
modules for the Schur algebra $S(\pi)$ corresponding to the poset
$\pi=\{(2|3)<(3'|4')<(4|5)\}$.  We begin with $T(4|5)$. By
\ref{Jantzen-iso} the tilting module $T(4|5) = P(2|3)$ is the
projective cover of $L(2|3)$.  By Proposition \ref{prop:BH-reciprocity}
and the Weyl module structure in Section \ref{sec:Weyl}, the
projective module $P(2|3)$ has a $\Delta$-filtration with
$\Delta$-factors $\Delta(2|3)$, $\Delta(3'|4')$, and $\Delta(4|5)$
each occurring with multiplicity one.  Using Proposition \ref{BGG} we
can locate where the heads of the $\Delta$-modules occur in a radical
filtration of $T(4|5)$.  The diagram below gives the radical structure
of the module
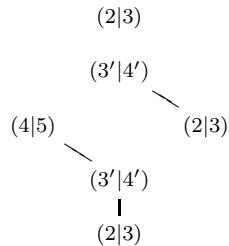
\begin{align*}
\begin{minipage}{60mm}
\def\objectstyle{\scriptstyle}\xymatrix@=6pt{
&		&(2|3)		&		\\
&		&(3'|4')\ar@{-}[dr]		&		\\
&(4|5)\ar@{-}[dr]		&		&(2|3)		\\
&		&(3'|4')\ar@{-}[d]		&		\\
&		&(2|3)		&		\\
}
\end{minipage} 
\end{align*}
in which the connected components are the layers in the
$\Delta$-filtration. Since $\rad_1(\Delta(2|3)) = L(2|3)$,
$\rad_2(\Delta(3'|4')) = L(2|3)$ and $\rad_3(\Delta(4|5))=L(2|3)$, by
Proposition \ref{BGG} the heads of the $\Delta$-modules $\Delta(2|3)$,
$\Delta(3'|4')$ and $\Delta(4|5)$ appear in the first, second, and
third layers of $P(2|3)$ respectively (as pictured above).  Note that
$\rad_4(P(2|3))=L(3'|4')$, however this module is not the head of a
$\Delta$-module in a $\Delta$-filtration.  Considering also the
$\nabla$-filtration gives the Alperin diagram
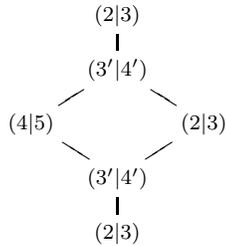
\begin{align*}
\begin{minipage}{60mm}
\def\objectstyle{\scriptstyle}\xymatrix@=6pt{
&		&(2|3)\ar@{-}[d]		&		\\
&		&(3'|4')\ar@{-}[dr]	\ar@{-}[dl]	&		\\
&(4|5)\ar@{-}[dr]		&		&(2|3)	\ar@{-}[dl]	\\
&		&(3'|4')\ar@{-}[d]		&		\\
&		&(2|3)		&		\\
}
\end{minipage}
\end{align*} 
of $T(4|5)$ shown above.  This is projective-injective and so
Proposition \ref{BGG} (and the subsequent remark) give both the
radical and socle structure. Hence $T(4|5)$ is rigid and the above
diagram is a strong Alperin diagram.

The above also gives us the structure of $T(3'|4')$, which appears as
a quotient of $T(4|5)$. To see this, notice that the module
$P(2|3)=T(4|5)$ has a uniserial quotient module isomorphic to $[(2|3),
  (3'|4'), (2|3)]$; this quotient has both $\Delta$ and $\nabla$
filtrations, hence is tilting and isomorphic to $T(3'|4')$.  Moreover,
since $\Delta(2|3)=L(2|3)=\nabla(2|3)$, the module $T(2|3)=L(2|3)$ is
a simple tilting module.

We now determine the structure of $T(6|8)$.  The calculation is
similar to that given above, so we only sketch it.  We have that
$\Delta(6|8)=[(6|8),(4|5)]$, so $\Delta(4|5)$ must appear at the top.
Contravariant duality and our knowledge of the other Weyl modules in
the  linkage class   give us the $\Delta$-filtration 
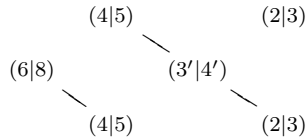
\begin{align*}
\begin{minipage}{60mm}
\def\objectstyle{\scriptstyle}\xymatrix@=6pt{
&&(4|5)\ar@{-}[dr]		&	&(2|3)		\\
&(6|8)\ar@{-}[dr]	&		&(3'|4')\ar@{-}[dr]		&		\\
&&(4|5)		&&(2|3)		&		\\
}
\end{minipage}
\end{align*}
of $T(6|8)$ as shown above. Finally, consideration of the
$\nabla$-filtration gives us the full structure diagram as
depicted in Theorem \ref{thm:2}.

We now consider $T(7|9)$.  
%%% By \cite[Theorem 5.1]{deV} it follows that
By \eqref{eq:tptilt} we have an isomorphism $T(7|9) \cong T(4'|5)
\otimes E^{(1)}$, and by comparing characters (or composition factors)
we see that $T(7|9) = P(4|5)$, as $S(\le 7|9)$-modules. Therefore
$T(7|9)$ is rigid by Proposition \ref{BGG}.  The full Alperin diagram
can then be deduced from the $\Delta$ and $\nabla$-filtrations.

\subsection{Coefficient quivers}\label{sec:cq}
To proceed further we need to use methods from the theory of finite
dimensional algebras. This need will come up again in Section
\ref{sec:p-reg}. Our approach will be similar to that of Ringel's
Appendix to \cite{BDM1}. In particular, when writing quiver and
relations we will switch to right modules instead of left ones, so
that composite paths can be read from left to right.  Recall that by
Gabriel's theorem (see e.g., Proposition 4.1.7 in \cite{Benson}) the
basic algebra of any finite dimensional algebra is isomorphic to a
suitable quotient of the path algebra of the ext-quiver of the
algebra.

We recall the terminology of coefficient quivers.  Let $K$ be a
field, let $Q$ be a (finite) quiver and $KQ$ the path algebra of $Q$
over $K$.  Recall that a representation $N$ of $Q$ over $K$ associates
to each vertex $i \in Q$ a vector space $N_i$ and to each arrow
$\alpha: i \to j$ a linear transformation $N_i \to N_j$. There is a
natural correspondence between representations of $Q$ and
$KQ$-modules. See \cite{ring1} or \cite{Benson} for more details.

Given a representation $N$, let $d_i$ be the dimension of $N_i$, and
$d = \sum_i d_i$; $d$ is called the dimension of $N$. A basis
$\mathcal{B}$ of $N$ is by definition a subset of the disjoint union
of the various $K$-spaces $N_i$, such that for any vertex $i$ the set
$\mathcal{B}_i = \mathcal{B} \cap N_i$ is a basis of $N_i$. Let us
assume that such a basis $\mathcal{B}$ of $N$ is given. For any arrow
$\alpha: i \to j$, we may express $N_\alpha$ as a $(d_j \times
d_i)$-matrix whose rows are indexed by $\mathcal{B}_j$ and whose
columns are indexed by $\mathcal{B}_i$. We denote by
$N_{\alpha,\mathcal{B}}(b,b')$ the corresponding matrix coefficients,
where $b \in \mathcal{B}_i$, $b' \in \mathcal{B}_j$, these matrix
coefficients $N_{\alpha,\mathcal{B}}(b, b')$ are defined by
$N_\alpha(b) = \sum_{b' \in
  \mathcal{B}}N_{\alpha,\mathcal{B}}(b,b')b'$.  By definition, the
coefficient quiver $\Gamma(N, \mathcal{B})$ of $N$ with respect to
$\mathcal{B}$ is the oriented graph with $\mathcal{B}$ as set of
vertices, and there is an arrow $\alpha: b \to b'$ provided
$N_{\alpha,\mathcal{B}}(b,b') \neq 0$. Usually that arrow would take
$\alpha$ as label, but since we deal only with quivers without
multiple arrows, the labels are omitted. We will always arrange our
coefficient quivers so that arrows point downwards; then arrows can be
replaced by edges.  It should be noted that the choice of basis can
affect the shape of the coefficient quiver.

\subsection{}\label{wallstuff}
We now consider the generalised Schur algebra $S(\le 8|9)=S(\pi)$, or
rather its basic algebra.  From the structure diagrams of the Weyl
modules in Section \ref{sec:Weyl} it follows that the $\Ext^1$-quiver
$P$ for $S(\le 8|9)$ with indexing set $\pi=\{1|2,4|6,4'|6',5|7,8|9\}$
is as illustrated in Figure \ref{figP}.

\begin{figure}[h]
\begin{center}
\begin{minipage}{30mm}  \xymatrix{		
		&		&	8|9  \ar@<2pt>[d]^{a}														\\
	&			& 5|7 \ar@<2pt>[dll]^{c_1'} \ar@<2pt>[drr]^{c_2'}  \ar@<2pt>[u]^{a'}	&					&	 \\
4|6   \ar@<2pt>[drr]^{b_1} \ar@<2pt>[urr]^{c_1}				&			& 							&			& 4'|6' \ar@<2pt>[dll]^{b_2} \ar@<2pt>[ull]^{c_2} \\
												&			&1| 2  \ar@<2pt>[ull]^{b_1'} \ar@<2pt>[urr]^{b_2'} 	&			& \\
}\end{minipage}
 \end{center}
\caption{The quiver $P$. }
 \label{figP}
\end{figure}
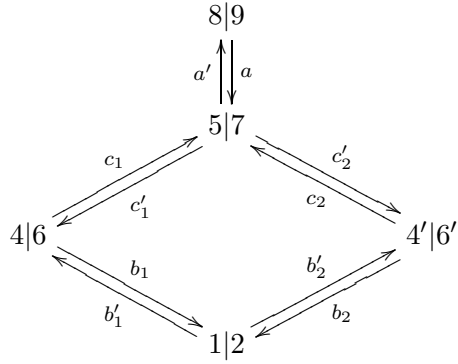

We label the idempotents corresponding to the nodes by $e_{1|2}$,
$e_{4|6}$, $e_{4'|6'}$, $e_{5|7}$, and $e_{8|9}$. By Gabriel's theorem
$S(\le 8|9)$ is Morita equivalent to a quotient of the path algebra of
$P$.

\begin{prop}\label{prop:S89}
  The basic algebra of the Schur algebra $S=S(\leq 8|9)$ is isomorphic to
  the path algebra of $P$ modulo the ideal generated by the following
  relations:
  \begin{gather*}  
  aa'=0, ac_2'=0, ac_1'b_1=0, ac_1'c_1=0, 
  c_i'c_i=   a'a,  
  c_1'b_1= c_2'b_2, 
  c_1'b_1b_2'=0,  
\\
  b_1b_1'= c_1c_1',    
  b_1b_2'= c_1c_2',    c_2a'=0,
  b_2b_1'= c_2c_1', 
     b_2b_2'= (-1)c_2c_2',  
  b_1'c_1 = b_2'c_2.
  \end{gather*}
%
%relations := [
%  # relations from P(8|9)
%  a*ap, a*c2p, a*c1p*b1, a*c1p*c1, 
%  # relations from P(5|7)
%  c1p*c1-ap*a, c2p*c2-ap*a, c1p*b1-c2p*b2, c1p*b1*b2p, 
%  # relations from P(4|6)
%  b1*b1p-c1*c1p, b1*b2p-c1*c2p, 
%  # relations from P(4'|6')
%  c2*ap, b2*b2p+c2*c2p, b2*b1p-c2*c1p, 
%  # relations from P(1|2)
%  b1p*c1+b2p*c2
%];
%
\end{prop}

\begin{proof}
By Proposition \ref{prop:BH-reciprocity}, the projective
indecomposable modules for $S(\leq 8|9)$ have the following
$\Delta$-filtrations (going downwards)
$$
\begin{array}{ccc}
P(1|2) 		&& \Delta(1|2) |  \Delta(4|6)\oplus  \Delta(4'|6') |  \Delta(5|7) \\
P(4|6) 		&& \Delta(4|6) |  \Delta(5|7)|\Delta(8|9) \\
P(4'|6') 		&& \Delta(4'|6') |  \Delta(5|7) \\
P(5|7) 		&& \Delta(5|7) | \Delta(8|9) \\
P(8|9) 		&&  \Delta(8|9) .
\end{array}
$$ By equation (\ref{Jantz}), we have the isomorphism $T(4|5)\cong
P(2|3)$ and by equation (\ref{eq:tptilt}) we have $T(8|9) \cong T(4|5)
\otimes E^{[1]}$.  This implies that $T(8|9)$ has simple head; it
therefore appears as a quotient of $P(4|6)$.  One can check that the
character of $T(4|5)\otimes E^{[1]}$ is equal to that of $P(4|6)$
given above, and therefore $T(8|9)=P(4|6)$ for $S(\leq 8|9)$.  By
equation (\ref{Jantz}), $P(1|2)$ is isomorphic to the tilting module
$T(5|7)$ for $S(\leq 8|9)$.

The Loewy layers of the module $T(8|9)$ are multiplicity-free and so
its structure is given by
$$ \begin{minipage}{60mm}
\def\objectstyle{\scriptstyle}
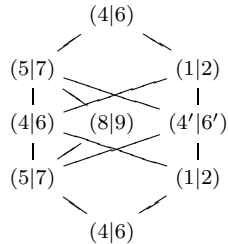
\xymatrix@=6pt{
						&(4|6)	\ar@{-}[dr]\ar@{-}[dl]		&	\\
(5|7)	\ar@{-}[drr]\ar@{-}[dr]\ar@{-}[d]	&					&(1|2)\ar@{-}[d]\ar@{-}[dll]\\
(4|6)					&(8|9)				&(4'|6')		\\
(5|7)	\ar@{-}[urr]\ar@{-}[ur]\ar@{-}[u]	&					&(1|2)\ar@{-}[u]\ar@{-}[ull]\\
			&(4|6)	\ar@{-}[ur]\ar@{-}[ul]			&		  }
 \end{minipage} $$
as determined by Proposition \ref{BGG} for both the $\Delta$ and
$\nabla$-filtrations. 

A tilting module $T(\lambda)$ has a unique submodule isomorphic to
$\Delta(\lambda)$ and a unique quotient module isomorphic to
$\nabla(\lambda)$. By definition (Section 4 of \cite[Appendix]{BDM1})
the \emph{core} of $T(\lambda)$ is $C(\lambda) =
Q(\lambda)/R(\lambda)$, where $R(\lambda)$ is the radical of
$\Delta(\lambda)$ and $Q(\lambda)$ is the kernel of the canonical
quotient map $T(\lambda) \to \nabla(\lambda)/L(\lambda)$. It is easily
checked that $R(\lambda) \subset Q(\lambda)$ and that $L(\lambda)$ is
a direct summand of the quotient $C(\lambda) = Q(\lambda)/R(\lambda)$.
From the structure of $T(8|9)$, we see that its core decomposes as a
direct sum of $L(8|9)$ and the module pictured in Figure \ref{b1b2}.
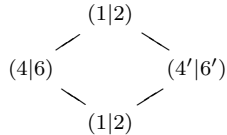
\begin{figure}[h]
\begin{minipage}{60mm} \def\objectstyle{\scriptstyle}
\xymatrix@=6pt{
 &					&(1|2) 	\ar@{-}[dr]	\ar@{-}[dl]					&		\\
&(4|6)	\ar@{-}[dr]				& 							&(4'|6')	\ar@{-}[dl]		\\
 &					&(1|2) 	  			&
  }
 \end{minipage} 
 \caption{The  non-simple direct summand of the core of $T(8|9)$.}
 \label{b1b2}
\end{figure}

Arguing by symmetry, we can obtain the structure for $T(8'|9')$ from
that of $T(8|9)$ above. This equals $P_{\pi'}(4'|6')$ for the
generalised Schur algebra $S(\pi') = S(\le 8'|9')$. It has a unique
submodule isomorphic to $\Delta(8'|9')$ and the corresponding quotient
module $T(8'|9')/\Delta(8'|9')$ is isomorphic to $P(4'|6')$ for $S(
\le 8|9)$. This determines the structure of $P(4'|6')$.  Furthermore,
by Proposition \ref{donkin} and considerations of the Loewy structure
of $P(5|7)$ we deduce that $T(5|7)$ embeds in $T(8|9)$.  Therefore,
the coefficient quivers of $P(5|7)$ and $P(4'|6')$ are as follows
$$ \begin{minipage}{60mm}
\def\objectstyle{\scriptstyle}
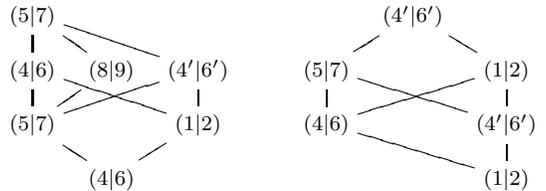
\xymatrix@=6pt{
(5|7)	\ar@{-}[drr]\ar@{-}[dr]\ar@{-}[d]	&					& \\
(4|6)					&(8|9)				&(4'|6')	\ar@{-}[dll] 	\\
(5|7)	 \ar@{-}[ur]\ar@{-}[u]	&					&(1|2)\ar@{-}[u]\ar@{-}[ull]\\
			&(4|6)	\ar@{-}[ur]\ar@{-}[ul]			&		  }
 \end{minipage}  \quad \quad
 \begin{minipage}{60mm}
\def\objectstyle{\scriptstyle}
\xymatrix@=6pt{
						&(4'|6')	\ar@{-}[dr]\ar@{-}[dl]		&	\\
(5|7)	\ar@{-}[drr] \ar@{-}[d]	&					&(1|2)\ar@{-}[d]\ar@{-}[dll]\\
(4|6)					& 				&(4'|6')		\\
  &					&(1|2)\ar@{-}[u]\ar@{-}[ull]	  }
 \end{minipage} .$$

\noindent
The diagrams allow us to immediately deduce when a coefficient in the
quiver is equal to zero.  The extensions in the diagram all contribute
non-zero coefficients.

From $P(8|9)=\Delta(8|9)=[8|9, 5|7, 4|6]$, we have that $aa'=ac_2'=0,
ac_1'b_1= ac_1'c_1=0$.
From the coefficient quiver for $P(5|7)$ we deduce that
$c_1'b_1=c_2'b_2 $, 
 $	c_1'b_1b_2'=0		$,  
and 
  $c_1'c_1=
  \alpha  
  a'a$,   
 $c_2'c_2=
  \beta a'a$,
   for some non-zero constants  
$\alpha, \beta  \in K$.

Consider the projective module $P(4|6)=T(8|9)$ (as pictured
above).  From the coefficient quiver for $P(4|6)$, we deduce that
$b_1b_j'= \gamma_j c_1c_j'$ for non-zero coefficients $\gamma_j\in K$
where $j=1,2$.  By examining $P(4'|6')$ in a similar fashion, we
deduce that $b_2b_j'= \delta_j c_2c_j'$ for non-zero coefficients
$\delta_j \in K$, but with the additional relation $c_2a'=0$.

% We now consider the structure of $P(1|2)$.  By self duality, we have that
%$b_1'b_1=
%\zeta
% b_2'b_2$ for a non-zero coefficient 
%$\zeta \in K$.  
% We   consider the submodules of $P(1|2)$ generated by the copies of $L(1|2)$ in the third Loewy layer; we shall use the self-duality of $P(1|2)$  and the  homomorphisms from other projective modules into $P(1|2)$, in order to deduce the values of the  coefficients $\alpha,\beta, \gamma_j,\delta_j,\zeta \in K$ for $j=1,2$.
% 
% 
%The module $\nabla(5|7)$  has coefficient quiver
%$$
%\begin{minipage}{60mm} \def\objectstyle{\scriptstyle}
%\xymatrix@=6pt{
% &					&(1|2) 	\ar@{-}[dr]	\ar@{-}[dl]					&		\\
%&(4|6)	\ar@{-}[dr]				& 							&(4'|6')	\ar@{-}[dl]		\\
% &					&(5|7) 	  			&
%  }
% \end{minipage} .
%$$
%Choosing a suitable basis of $\nabla(5|7)$, we can assume that at least 3 of the non-zero
%coefficients are equal to 1 and we look at the remaining coefficient, say that for the
%arrow $c_2$. It will be a non-zero scalar $k \in K$. Recall that we have started with a
%particular generator choice for the algebra $S$ which we can change. If we replace
%the element $c_2$ by $(1/k)c_2$, then the coefficients needed for  $\nabla(5|7)$ will all be equal to 1. 
% Similarly, the  coefficients, $\alpha, \beta$  can  all   to be chosen to be equal to 1.   
%
%It remains to consider $\gamma_j,\delta_j\in K$ for $j=1,2$. By Proposition \ref{donkin} and considerations of Loewy structure, we deduce that there is an injective map, $f_1$, (respectively $f_2$) from $P(4'|6')$ (respectively $P'(4|6)$) to a submodule  of $P(1|2)$.  

Consider the structure of the quotient $\nabla(5|7)$ of
$T(5|7)=P(1|2)$.  By self duality, we have that $b_1'b_1= \zeta
b_2'b_2$ for a non-zero coefficient $\zeta \in K$.  The module
$\nabla(5|7)$ has coefficient quiver
$$
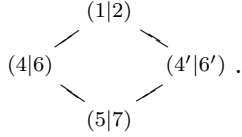
\begin{minipage}{60mm} \def\objectstyle{\scriptstyle}
\xymatrix@=6pt{
 &					&(1|2) 	\ar@{-}[dr]	\ar@{-}[dl]					&		\\
&(4|6)	\ar@{-}[dr]				& 							&(4'|6')	\ar@{-}[dl]		\\
 &					&(5|7) 	  			&
  }
 \end{minipage} .
$$ Choosing a suitable basis of $\nabla(5|7)$, we can assume that at
least 3 of the non-zero coefficients are equal to 1 and we look at the
remaining coefficient, say that for the arrow $c_2$. It will be a
non-zero scalar $k \in K$. Recall that we have started with a
particular generator choice for the algebra $S(\le 8|9)$ which we can
change. If we replace the element $c_2$ by $(1/k)c_2$, then the
coefficients needed for $\nabla(5|7)$ will all be equal to 1.
Similarly, the coefficients $\alpha, \beta$ can to be chosen to be
equal to 1.

Now consider the submodules of $P(1|2)$ generated by the copies of
$L(1|2)$ in the third Loewy layer; we shall use the self-duality of
$P(1|2)$ and the homomorphisms from other projective modules into
$P(1|2)$, in order to deduce the values of the coefficients
$\gamma_j, \delta_j \in K$ for $j=1,2$.

% We let $\pi'$ denote the idempotent corresponding to the subset of
%weights $\pi/\{8|9\}$. Applying the idempotent truncation map we
%obtain an $S(\pi)$-module $P'(4|6) =e_{\pi'}P(4|6)$.  (The notation
%reflects that we have trivially inflated the corresponding projective
%module for $S(\pi')$).

The module $P(4|6) = T(8|9)$ has a unique submodule isomorphic to
$\Delta(8|9)$.  We let $P'(4|6)$ denote the corresponding quotient
module $P(4|6)/\Delta(8|9)$.  (The notation reflects that we have
trivially inflated the corresponding projective module for $S(\pi')$
for $\pi'=\{(1|2), (4|6), (4'|6'), (5|7)\}$).  By Proposition
\ref{donkin} and considerations of Loewy structure, we deduce that
there is an injective map $f_1$ (respectively $f_2$) from $P(4'|6')$
(respectively $P'(4|6)$) to a submodule of $P(1|2)$.  In what follows
we shall identify a simple composition factor of a projective module
with the path in the quiver which terminates at the given simple
composition factor.  An example of how to pass between these two
pictures (the coefficient quiver and the subspace lattice) is given in
\cite[Appendix page 217]{BDM1}.

The injection $f_1$ (respectively, $f_2$) takes the simple head of
$P(4'|6')$ (respectively, of $P'(4|6)$) which is labelled by the
element $e_{4'|6'}$ (respectively $e_{4|6}$) to the simple composition
factor $L(4'|6')$ in the second radical layer of $P(1|2)$, which is
labelled by the path $b_1'$ (respectively $b_2'$). Therefore $f_1$
(respectively $f_2$) takes the simple composition factor $L(1|2)$ in
the second radical layer of $P(4'|6')$ (respectively $P'(4|6)$)
labelled by $b_1$ (respectively $b_2$) to the corresponding simple
factor $L(1|2)$ in the second radical layer of $P(1|2)$ labelled by
$b_1'b_1$ (respectively $b_2'b_2$).  From the diagram of $P(4'|6')$
(respectively $P'(4|6)$) we know that the simple composition factor
labelled by $b_2$ (respectively $b_1$) generates the module with
structure as given in Figure \ref{b1b2}.  We therefore deduce that the
copy of $L(1|2)$ in the third Loewy layer of $P(1|2)$ labelled by the
path $b_1'b_1$ (respectively $b_2'b_2$) generates a submodule
isomorphic to the module given in Figure \ref{b1b2}.

By the self-duality of $P(1|2)$ we know that $[(1|2),(4|6),(1|2)]$ and
$[(1|2),(4'|6'),(1|2)]$ must also appear as submodules of $P(1|2)$ and
are therefore generated by diagonal embeddings of $L(1|2)$ into the
third radical layer of $P(1|2)$, these are labelled by linear
combinations of the paths $b_1'b_1$ and $b_2'b_2$.  Rescaling the
generators if necessary, we may choose $b_1'b_1+b_2'b_2$ (respectively
$b_1'b_1-b_2'b_2$) as the path labelling the diagonal copy of $L(1|2)$
which generates the submodule $[(1|2),(4|6),(1|2)]$ (respectively
$[(1|2),(4'|6'),(1|2)]$).

The submodule of $P(1|2)$ generated by the copy of $L(1|2)$ labelled
by $b_1'b_1$ has composition factors labelled by the following paths
\begin{align*}
		& b_1'b_1 					\\
&b_1' (b_1 b_1')  
 =\gamma_1
 b_1'	c_1c_1'	\\
&  b_1' (b_1 b_2')
 =\gamma_2 b_1'	c_1c_2'	\\ 					
&
 b_1'(b_1b_1')b_1 =
 \gamma_1 b_1'c_1c_1'b_1  ,
\intertext{and the submodule generated by the copy of $L(1|2)$
  labelled by $b_2'b_2$ has composition factors labelled by the
  following paths}
  		& b_2'b_2 					\\
&b_2'b_2b_1'= 
\delta_1 b_2'	c_2c_1' =
\delta_1\gamma_1  
  b_1'	c_1c_1' 
  	\\
&b_2'b_2b_2'= 
\delta_2 b_2'	c_2c_2' =
 \delta_2\gamma_2  b_1'	c_1 c_2' 
  	\\
&(b_2'b_2)(b_2'b_2)=
 \zeta^2(b_1'b_1)(b_1'b_1)  =\zeta^2b_1'(b_1b_1')b_1
 =\zeta^2 \gamma_1 b_1'c_1c_1'b_1 .
 \end{align*}
From our discussion of the embeddings $f_1$ and $f_2$, we deduce that
all of these paths are non-zero.

The diagonal copy of $L(1|2)$ labelled by the path $b_1'b_1+ b_2'b_2$
generates the submodule $[(1|2),(4|6),(1|2)]$.  Therefore, we require
that the coefficients satisfy the following identities
$$ \begin{array}{llllllll} \gamma_1 + \delta_1\gamma_1 \neq 0
  &&\gamma_2+\delta_2\gamma_2=0 \end{array}.$$ Here, the left hand
side of these (in)equalities is the coefficient of the path labelling
the composition factor $L(4|6)$ (respectively $L(4'|6')$) in the
submodule generated by the diagonal copy of $L(1|2)$ labelled by the
path $b_1'b_1+ b_2'b_2$.  We require this coefficient to be non-zero
(respectively zero) in order for the submodule generated by the copy
of $L(1|2)$ labelled by the path $b_1'b_1+ b_2'b_2$ to be of the form
$[(1|2),(4|6),(1|2)]$.

Similar considerations in
the case $b_1'b_1- b_2'b_2$ imply the (in)equalities
$$ \begin{array}{llllllll}\gamma_1 - \delta_1\gamma_1 = 0
  &&\gamma_2-\delta_2\gamma_2 \neq 0 \end{array}.$$ This implies that
the coefficients satisfy $\delta_1=1$ and $\delta_2=-1$.  Rescaling if
necessary, we may choose to take $\gamma_1=\gamma_2=1$.  This
completes the proof.
\end{proof} 
 
\begin{rmk}
  Using the``Quivers and Path Algebras'' package \cite{QPA} for GAP
  \cite{GAP}, we have checked that the algebra described by quiver and
  relations in Proposition \ref{prop:S89} does have the indicated
  dimension. We have also checked that the dimensions and Loewy layers
  of the indecomposable projectives computed by the package agree with
  our results.
\end{rmk}

We provide the two most symmetric coefficient quivers of $T(5|7)$
below.  
 $$  \begin{minipage}{30mm}
\def\objectstyle{\scriptstyle}
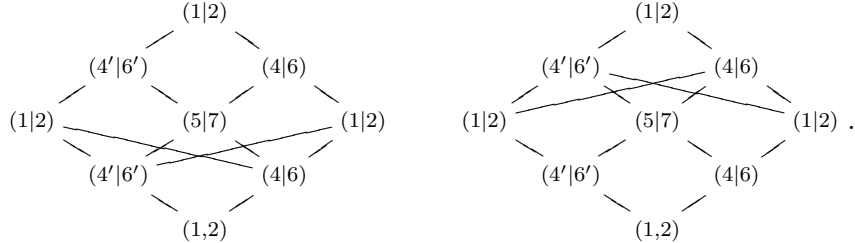
\xymatrix@=6pt{
		&	 								& (1|2)	  \ar@{-}[dl] 	 \ar@{-}[dr] 	& 					&\\
		&(4'|6') \ar@{-}[dl] 	 \ar@{-}[dr] 			&		    					& (4|6) \ar@{-}[dr] \ar@{-}[dl] 		& \\
		(1|2)\ar@{-}[dr] 	 \ar@{-}[drrr] &									& (5|7)  \ar@{-}[dr] \ar@{-}[dl]	& 	& (1|2) \ar@{-}[dl]  \ar@{-}[dlll] 	\\
		&(4'|6')							& 				  			&(4|6) \\
		&									& ( 1,2) \ar@{-}[ur] \ar@{-}[ul] 		&
}
\end{minipage}
\quad \quad
\begin{minipage}{30mm}
\def\objectstyle{\scriptstyle}
\xymatrix@=6pt{
		&	 								& (1|2)	  \ar@{-}[dl] 	 \ar@{-}[dr] 	& 					&\\
		&(4'|6')  \ar@{-}[drrr] \ar@{-}[dl] 	 \ar@{-}[dr] 			&		    					& (4|6) \ar@{-}[dr] \ar@{-}[dl]  \ar@{-}[dlll] 		& \\
		(1|2)\ar@{-}[dr] 	&									& (5|7)  \ar@{-}[dr] \ar@{-}[dl]	& 	& (1|2) \ar@{-}[dl] 	\\
		&(4'|6')							& 				  			&(4|6) \\
		&									& ( 1,2) \ar@{-}[ur] \ar@{-}[ul] 		&
}
\end{minipage}.
 $$ 
The only choice to be made is which basis to take for the
2-dimensional space $L(1|2)\oplus L(1|2)$ in the third Loewy layer.
The left coefficient quiver corresponds to the basis $N=\{b_1'b_1,
b_2'b_2 \}$ and the right one corresponds to the basis
$N'=\{b_1'b_1+b_2'b_2,b_1'b_1-b_2'b_2 \}$.

\section{The $p$-regular tilting modules}\label{sec:p-reg}\noindent
Having determined the structure of the $p$-singular tilting modules,
we now turn to the $p$-regular ones.  This will complete the proof of
Theorem \ref{thm:2}.  The structure of the tilting modules $T(1)$,
$T(2)$, and $T(3)$ is easily verified, and is left to the reader, but
$T(4)$ is more complicated.
 
\subsection{}
%The dimension of $\Ext^1_G$ between simple modules for an algebraic
%group $G$ is determined by the structure of the Weyl modules (see
%Proposition II.4.14 in \cite{Jantzen}).  
We first consider the module $T(6)$ as this will be helpful in
determining the structure of $T(4)$.  This module can be seen to be
projective-injective for $S(\le 6)$ by application of translation
functors to the embedding $T(4|6) \hookrightarrow T(5|7)$ (and use of
Proposition \ref{prop:BH-reciprocity} and \ref{Jantzen-iso}).  It is
therefore equal to the projective cover of $L(1)$.  It then follows by
Proposition \ref{BGG} that $T(6)$ is rigid.
%, with $\Delta$ and $\nabla$ filtrations given below.
%$$
% \begin{minipage}{34mm}
%\def\objectstyle{\scriptstyle}\xymatrix@=6pt{
%&				& 											&1 	 			 	  \\
%  &				& 4 	 \ar@{-}[d]\ar@{-}[dl]\ar@{-}[drr]		&   									&2 \ar@{-}[dr]		  \\
%&1\ar@{-}[dr]  	&3 \ar@{-}[d]								&6 		\ar@{-}[dr]	&3' 					&1  	 \\
%  &		&2  \ar@{-}[urr]						& 								&4 \ar@{-}[dl]	 		 \\
%   &		& 								&1 								& 					 
%}
% \end{minipage}
% \quad
%  \begin{minipage}{34mm}
%\def\objectstyle{\scriptstyle}\xymatrix@=6pt{
%&				& 											&1 	\ar@{-}[dl] 		 	  \\
%  &				& 4%\ar@{-}[dr] 	\ar@{-}[dr]\ar@{-}[d]\ar@{-}[dl]\ar@{-}[drr]	
%  	&   									&2 \ar@{-}[dll]\ar@{-}[d]\ar@{-}[dr]		  \\
%&1\ar@{-}[dr]  	&3  							&6 \ar@{-}[ul]					&3'  					&1 	 \\
%  &		&2 \ar@{-}[ul]						& 								&4	\ar@{-}[ull]\ar@{-}[u]	\ar@{-}[ru]			 \\
%   &		& 								&1 								& 					 
%}
% \end{minipage}
% $$
%Clearly the $L(2)$ and $L(4)$ extend 
%%While $L(1)$ does  appear with multiplicity two,
%% this module is not as complicated as $T(4)$.  
The diagram of $T(6)$ in Theorem \ref{thm:2} is obtained using the
equalities $\dim_K \Ext_{G}^1(L(2),L(1)) = 1 = \dim_K
\Ext_{G}^1(L(4),L(1))$ coming from the structure of the Weyl modules,
and by reconciling the $\Delta$-filtration
%with sections $\Delta(4)$, $\Delta(2)$ 
with the corresponding $\nabla$-filtration.

\subsection{} 
To determine the structure of $T(4)$ we will study the generalised
Schur algebra $S(\le 4)$.  From the structure diagrams of the Weyl
modules in Section \ref{sec:Weyl} it follows that the $\Ext^1$-quiver
$Q$ for $S(\le4) = S(\pi)$ for $\pi = \{1,2,3,3',4\}$ is as
illustrated in Figure \ref{figQ}.
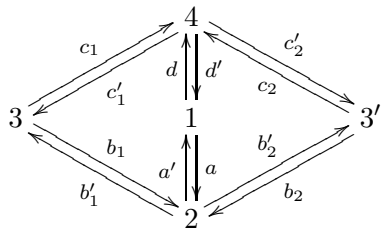
\begin{figure}[h]
\begin{center}
\begin{minipage}{34mm}  \xymatrix{		
												&			& 4 \ar@<2pt>[dll]^{c_1'} \ar@<2pt>[drr]^{c_2'} \ar@<2pt>[d]^{d'}	&					&	 \\
3  \ar@<2pt>[drr]^{b_1} \ar@<2pt>[urr]^{c_1}				&			&1 \ar@<2pt>[d]^{a} \ar@<2pt>[u]^{d}								&			& 3'\ar@<2pt>[dll]^{b_2} \ar@<2pt>[ull]^{c_2} \\
												&			& 2\ar@<2pt>[u]^{a'} \ar@<2pt>[ull]^{b_1'} \ar@<2pt>[urr]^{b_2'} 	&			& \\
}\end{minipage}
 \end{center}
\caption{The quiver $Q$}
 \label{figQ}
\end{figure}
By Gabriel's theorem (see
e.g.\ Proposition 4.1.7 in \cite{Benson}) we have that $S(\le 4)$ is a
quotient of the path algebra of the quiver $Q$.
Notice that applying appropriate translation functors to the embedding
$T(3|4) \hookrightarrow T(4'|5)$ produces an embedding $T(4)
\hookrightarrow T(5)$; see \cite[II.E.11]{Jantzen}.  By Propositions
\ref{prop:BH-reciprocity} and \ref{donkin} it follows that $T(4)$ is
the projective cover of $L(2)$ as an $S(\le 4)$-module.

Before considering the defining relations of $S(\pi)=S(\le 4)$ we
first consider the simpler question of describing the Schur algebra
$S' = S(\pi')$ for $\pi'=\{1,2,3,3'\}$, which is a quotient algebra of
$S(\le 4)$, by quiver and relations. From the Weyl module structure,
the $\rm{Ext}^1$-quiver for $S'$ is the full subquiver
$Q'=Q(1,2,3,3')$ of $Q$ obtained by removing the vertex $4$ and the
arrows $c_1, c_1', c_2, c_2', d, d'$ starting or terminating at that
vertex.  It will soon become necessary to compare projective
indecomposables for $S'$ with those for $S$. Our notational convention
is to use $P(j) = P_\pi(j)$ for the projective cover of $L(j)$ in the
algebra $S = S(\pi)$, and $P'(j) = P_{\pi'}(j)$ for the corresponding
projective cover in $S' = S(\pi')$. We have the following description
of the algebra $S'$.

\begin{prop}\label{T4}
  The basic algebra of the Schur algebra $S'=S(\pi')$ is isomorphic to
  the path algebra of $Q'$ modulo the ideal generated by the following
  relations:
  \begin{center}
  $ab'_i=0 ,\ \ aa'a=0,\ \ b_ia'=b_ib_i'=b_ib_j'=0,
    \ \ a'a=b_1'b_1+b_2'b_2$.
 \end{center}
  where $i\neq j$.
\end{prop}

\begin{proof}
By Proposition \ref{prop:BH-reciprocity}, the projective indecomposable
modules for $S'$ have the following $\Delta$-factors (going
downwards)
\[
\begin{array}{ccc}
P'(1) 		&& \Delta(1) |  \Delta(2) \\
P'(2) 		&& \Delta(2) |  \Delta(3)\oplus \Delta(3') \\
P'(3) 		&& \Delta(3) \\
P'(3') 	&& \Delta(3')  .
\end{array}
\]
From the structure of the Weyl modules and Proposition \ref{BGG} it
follows that $P'(1)$, $P'(3)$, and $P'(3')$ are uniserial with
structure $P'(1)=[1, 2, 1]$, $P'(3) = \Delta(3) = [3,2]$, $P'(3') =
\Delta(3') = [3',2]$. This implies the first three relations in the
proposition.

We now address the remaining relation $a'a=b_1'b_1 + b_2'b_2$. For
this we need structural information about $P'(2)$. Proposition
\ref{BGG} gives the Loewy structure of $P'(2)$ as pictured below
\[
  \begin{minipage}{34mm}
   \def\objectstyle{\scriptstyle}\xymatrix@=6pt{
      &		&2 \ar@{-}[d]		&\\
      &3\ar@{-}[d]		&1 	&3'\ar@{-}[d]\\
      &2		&		&2\\
      }
  \end{minipage}
\]
which has a $\Delta$-filtration with subquotients isomorphic to
$\Delta(3)$, $\Delta(3')$ and $\Delta(2)$.  It is immediate that
$a'a=\beta_1b_1'b_1+\beta_2b_2'b_2$ where $\beta_1$, $\beta_2$ are
scalars, since otherwise the independence of the paths $a'a$,
$b_1'b_1$, $b_2'b_2$ would force $[P'(2)\colon L(2)] > 3$, which is a
contradiction.

By Proposition \ref{donkin} we have $\dim_K \Hom_{G}(P'(2),T(3))=2$.
The Loewy structure of $P'(2)$ implies that one of the two
homomorphisms is given by projection of the head of $P'(2)$ onto the
socle of $T(3)$, and the other homomorphism is a surjection of $P'(2)$
onto $T(3)$.  Therefore $T(3)$ is a quotient module of $P'(2)$ and
$\beta_1 \neq 0$.  We also have that $\beta_2 \neq 0$ by symmetry.
Finally, fixing our choice of $a,a',b_1,b_2$ and adjusting our choice
for $b_1',b_2'$ if necessary, we can pick $\beta_1$ and $\beta_2$ to
both be 1. This concludes the proof.
\end{proof}

\subsection{}\label{M}
\emph{Definition.} Let $M(2)$ denote the quotient module defined by
\begin{align*}
  M(2) = P'(2)/(b_1'b_1-b_2'b_2).
\end{align*}
It is clear that $M(2)$ has strong Alperin diagram
\[ \scriptstyle M(2) = 
\begin{minipage}{30mm}
\def\objectstyle{\scriptstyle}\xymatrix@=6pt{
&		&2\ar@{-}[dr]\ar@{-}[dl]\ar@{-}[d]		&\\
&3\ar@{-}[dr]		&1\ar@{-}[d]		&3'\ar@{-}[dl]\\
&		&2		&		&			\\	
}
\end{minipage}
\]
as already mentioned in \ref{ss:M}. For any $\lambda \in C_2$ we
therefore have a module $M(\lambda)$ with similar structure, as
described in \ref{ss:M}.

%We have that $M(2)$ is the only non-simple,
%contravariantly self-dual quotient of $P'(2)$.  
%Any contravariantly self-dual quotient of $P'(2)$ has precisely one
%copy of $L(2)$ in the socle, and therefore corresponds to a quotient

We are interested in the structure of $T(4)$ and hence wish to consider the modules which appear 
as both quotients \emph{and} submodules of $T(4)$.  This leads us to consider the  quotient modules of $P'(2)$ which have simple socle isomorphic to $L(2)$.
These are given by taking the quotients corresponding to setting $b_2'=0$, $b_1'=0$, $b_1'b_1-
b_2'b_2=0$ , $b_1'b_1+ b_2'b_2=0$ or $a'=0$. The resulting modules
have coefficient quivers depicted in Figure \ref{lotsapics}.  Each of
these five modules has a corresponding coefficient quiver, with basis
$\mathcal{B}_i$ for $i=1,\ldots ,5$ respectively.
 
%  All but two of the coefficients can all be chosen to be equal to 1; when we quotient by the relation $b_1'b_1+b_2'b_2=0$ we have in the resulting quotient  that $b_1'b_1 =(-1) b_2'b_2$  and similarly in the case $a'=0$   (notice that the case that $a'=0$ is a quotient of the case where $b_1'b_1+b_2'b_2=0$).
 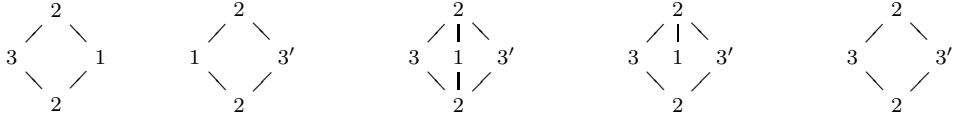
\begin{figure}[ht]
 $$ 
\begin{minipage}{30mm}
\def\objectstyle{\scriptstyle}\xymatrix@=6pt{
&		&2\ar@{-}[dr]\ar@{-}[dl] 	&\\
&3\ar@{-}[dr]		& 		&1\ar@{-}[dl]\\
&		&2		&		&			\\	
}
\end{minipage}
% \quad
\begin{minipage}{30mm}
\def\objectstyle{\scriptstyle}\xymatrix@=6pt{
&		&2\ar@{-}[dr]\ar@{-}[dl] 	&\\
&1\ar@{-}[dr]		& 		&3'\ar@{-}[dl]\\
&		&2		&		&			\\	
}
\end{minipage}
 \quad
\begin{minipage}{30mm}
\def\objectstyle{\scriptstyle}\xymatrix@=6pt{
&		&2\ar@{-}[dr]\ar@{-}[dl]\ar@{-}[d]		&\\
&3\ar@{-}[dr]		&1\ar@{-}[d]		&3'\ar@{-}[dl]\\
&		&2		&		&			\\	
}
\end{minipage}
\quad
\begin{minipage}{30mm}
\def\objectstyle{\scriptstyle}\xymatrix@=6pt{
&		&2\ar@{-}[dr]\ar@{-}[dl]\ar@{-}[d]		&\\
&3\ar@{-}[dr]		&1 		&3'\ar@{-}[dl]\\
&		&2		&		&			\\	
}
\end{minipage}
\quad
\begin{minipage}{30mm}
\def\objectstyle{\scriptstyle}\xymatrix@=6pt{
&		&2\ar@{-}[dr]\ar@{-}[dl] 		&\\
&3\ar@{-}[dr]		&  		&3'\ar@{-}[dl]\\
&		&2		&		&			\\	
}
\end{minipage}
$$
 \caption{The coefficient quivers of 
   quotient modules of $P'(2)$ corresponding $b_2'=0$, $b_1'=0$,
   $b_1'b_1- b_2'b_2=0$, $b_1'b_1 + b_2'b_2=0$ or $a'=0$ respectively.
   Numbering from left to right, let $\mathcal{B}_i$ denote the basis
   of the $i$th coefficient quiver for $i=1,\ldots, 5$.  The
   non-trivial coefficients are $N_{a,\mathcal{B}_3}(a',b_1'b_1)=2$,
   $N_{b_2,\mathcal{B}_4}(b_2',b_1'b_1)=-1$, and
   $N_{b_2,\mathcal{B}_5}(b_2',b_1'b_1)=-1$.  All other coefficients
   may be chosen to be equal to 1.}
 \label{lotsapics}
\end{figure}

%%%%%%%%%%%%%%%%%%%%%%%%%%%%%%%%%%%%%%%%%%%%%%%%%%%%%%%%%%%%%%%%%%%%%%
%% Eliminated, since this is already
%% explained in the caption of Figure 7 above. --SRD
%%%%%%%%%%%%%%%%%%%%%%%%%%%%%%%%%%%%%%%%%%%%%%%%%%%%%%%%%%%%%%%%%%%%%%
%In the first two coefficient quivers, the coefficients
%$N_{\alpha}(e_i,e_j)$ can be chosen to be equal to 1 for all for $i,j
%\in \{1,2,3,3',4\}$ and $\alpha:i\to j$.  In each of the final three
%quivers, all but one of the coefficients in $\mathcal{B}_i$ for
%$i=3,4,5$ may be chosen to be equal to 1.  For the module $M(2)$ (the
%third module with basis $\mathcal{B}_3$), we have that
%$a'a=b_1'b_1+b_2'b_2=2b_1'b_1$ modulo the relation $b_1'b_1-b_2'b_2=0$
%and so we have a non-trivial coefficient
%$N_{a,\mathcal{B}_3}(e_1,e_2)=2$.  We have that $b_1'b_1 =(-1)
%b_2'b_2$ modulo the relation $b_1'b_1+b_2'b_2=0$; therefore there is
%precisely one negative coefficient in each of the final two
%coefficient quivers, $N_{b_1,\mathcal{B}_4}(e_3,e_2)=-1$, and
%$N_{b_1,\mathcal{B}_5}(e_3,e_2)=-1$.
% (Here we have implicitly used the fact that the the fifth module is
% a quotient of the fourth).
%%%%%%%%%%%%%%%%%%%%%%%%%%%%%%%%%%%%%%%%%%%%%%%%%%%%%%%%%%%%%%%%%%%%%%%

\subsection{}\label{732} 
We now turn our attention to computing the defining relations for
$S(\le 4) = S(\pi)$ with $\pi =\{1,2,3,3',4\}$. The algebra $S(\le 4)$
is the direct sum of its projective indecomposables $P(1)$,
$P(2)$, $P(3)$, $P(3')$ and $P(4)$ and the
corresponding projective indecomposables in $S'$ are homomorphic
images of these. By character considerations or otherwise it is easy
to see that $P(2) = T(4)$. We will soon need the following fact.

\begin{lem}\label{LEMER}
The tilting module $T(4)$ has three filtrations as depicted in the
diagrams below
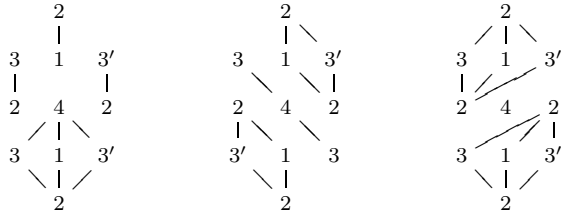
\[
 \begin{minipage}{34mm}
   \def\objectstyle{\scriptstyle}\xymatrix@=6pt{
      &		&2 \ar@{-}[d]		&\\
      &3\ar@{-}[d]		&1 	&3'\ar@{-}[d]\\
      &2		&4\ar@{-}[dr]\ar@{-}[dl]\ar@{-}[d]		&2\\
      &3\ar@{-}[dr]		&1\ar@{-}[d]		&3'\ar@{-}[dl]\\
      & &2 & & }
  \end{minipage}
  \quad
       \begin{minipage}{34mm}\def\objectstyle{\scriptstyle}\xymatrix@=6pt{
      &		&2\ar@{-}[dr]\ar@{-}[d]		&\\
      &3\ar@{-}[dr]		&1\ar@{-}[dr]		&3'\ar@{-}[d]\\
      &2\ar@{-}[dr]\ar@{-}[d]		&4		&2 		\\
      &3'\ar@{-}[dr]		&1\ar@{-}[d]		&3\ar@{-}[ul]\\
      & &2 & & }\end{minipage}
      \quad
      \begin{minipage}{34mm}\def\objectstyle{\scriptstyle}\xymatrix@=6pt{
      &		&2\ar@{-}[dr]\ar@{-}[dl]\ar@{-}[d]		&\\
      &3\ar@{-}[dlr]		&1\ar@{-}[dl]		&3'\ar@{-}[dll]\\
      &2		&4		&2\ar@{-}[dlr]\ar@{-}[dll]\ar@{-}[dl]		\\
      &3\ar@{-}[dr]		&1\ar@{-}[d]		&3'\ar@{-}[dl]\\
      & &2 & & }
 \end{minipage}
\]
in which the connected components in each diagram identify successive
subquotients of the filtrations.  
%There is no analogue of the second
%filtration with the roles of $3$ and $3'$ interchanged.
\end{lem}

\begin{proof}
The first filtration, depicted in the leftmost diagram above, is a
$\Delta$-filtration whose structure is determined by Proposition
\ref{BGG}.

To obtain the second filtration, first pick a weight $\nu\in
\mathcal{F}_{(3|4)}$ so that $\nu+ \overline{\varepsilon}_1\in C_{4}$.
(See \ref{ss:8:notation} for the notation $\overline{\varepsilon}_1$.)
In what follows, fix all alcove weights to be elements of the linkage
classes $\nu+\overline{\varepsilon}_j$ for $j = 1,2,3$ (There are only
two such linkage classes involved, as $\nu+\overline{\varepsilon}_1$
is linked to $\nu+\overline{\varepsilon}_3$.)  Note that
$T(3|4)=[(2|3'),(3|4),(2|3')]$ is uniserial.  Therefore the
$p$-regular linkage class of $E \otimes T(3|4)$ has a filtration with
layers given by the $p$-regular linkage classes of $E \otimes
L(2|3')$, $E \otimes L(3|4)$ and $E \otimes L(2|3')$.  Since $L(2|3')$
is tilting, it follows by character considerations that the
$p$-regular linkage class of $E \otimes L(2|3')$ is equal to
$T(3')$. This justifies the top and bottom connected components in the
middle diagram above. It remains to show that the $p$-regular linkage
class of $E \otimes L(3|4)$ is uniserial.  It is enough to show that
\[
\dim_K(  \Hom(L(3), E \otimes L(3|4))) =1.
\] 
This is equivalent to showing that $\dim_K( \Hom(L(3) \otimes E^*, L(3|4)))=1,$
which is easily seen to hold, as $L(3) \otimes E^*$ has exactly one
composition factor isomorphic to $L(3|4)$.

%We now wish to show that there does not exist a uniserial module of
%the form $[3',4,3']$.  If such a module exists, it is a quotient of
%the projective module $P_{\pi}(3')$.  We have that
%$\Ext^1(\Delta(3),\Delta(4))=K=\Ext^1(\Delta(3'),\Delta(4))$ and
%$\Ext^1(\Delta(2),\Delta(3))=K=\Ext^1(\Delta(2),\Delta(3'))$.
%Therefore we have an exact sequence
% \[
%0\to P_{\pi}(3') \to T(4) \to T(3) \to 0
%\]
%therefore there exists a uniserial module of the form $[3',4,3']$ if
%and only if it appears as a subquotient of $T(4)$.  The module
%$[3',4,3']$ is a subquotient of $T(4)$ if and only if the there is a
%sequence
% \[ 
%T(3)  \hookrightarrow T(4) \twoheadrightarrow T(3).
%\]
%We will show the existence of such a sequence leads to a contradiction.
% 
%
%To see this we consider the specific linkage class with highest weight
%$(p+1, p-2) \in C_4$ so that $E\otimes T(4)$ has a linkage class
%component isomorphic to $T(5|6)\oplus T(2|3)$ and the corresponding
%linkage class component of $T(3')\otimes E$ is isomorphic to
%$T(3'|4')$.  We tensor the sequence above with $E$ and project onto
%the unique  linkage class   with highest weight in $\mathcal{F}_{4|5}$ to obtain
%the sequence
%\[
%T(3'|4')\hookrightarrow T(4|5) \twoheadrightarrow T(3'|4')
%\]
%\textcolor{red}{[Why is tensoring exact in this situation?]}
%This gives the desired contradiction, by the structure (or indeed the
%character) of $T(4|5)$ given in Section \ref{BGGeg}.
%
%
%

To obtain the third filtration, we switch temporarily to highest
weight notation, and consider for example the tensor product
$T(p-2,p-2) \otimes T(p,0)$, which is tilting, hence a direct sum of
indecomposable tilting modules, with $T(2p-2,p-2) =T(4)$ occurring
with multiplicity one.  Since $T(p-2,p-2) = [K, L(p-2,p-2), K]$ we see
that the tensor product has a filtration with layers
\[
\begin{array}{cc}
  T(p,0) \\ \hline  L(p-2,p-2) \otimes L(p-2,1)\\ \hline  
L(p-2,p-2) \otimes ( 	L(p,0) \oplus   L(p-3,0) )	\\ 
\hline   L(p-2,p-2) \otimes L(p-2,1) \\  \hline 	
T(p,0) %]				
 \end{array}
\]
where every layer is contravariantly self-dual (as it is a tensor
product of contravariantly self-dual modules).  The simple module
$L(2p-2,p-2) = L(p-2,p-2) \otimes L(p,0)$ of highest weight appears as
a direct summand of the third layer; it must extend both above and
below to result in a module isomorphic to $T(2p-2,p-2)$, i.e. there
must exist modules $N_1$ and $N_2$ which are direct summands of
$T(p,0)$ and $L(p-2,p-2)\otimes L(p-2,1)$ respectively such that $
N_1|N_2|L(4)|N_2|N_1$ is a filtration of $T(4)$. We do not insist that
both $N_1$ and $N_2$ are non-zero.

By assumption, $N_1$ is either equal to $T(p,0)$ or zero.  
%There are two possible contravariantly self-dual modules which could
%occur as direct summands in the second/fourth layers which have the
%correct characters. 
Assume that $N_1=T(p,0)$ for a contradiction.  Then (by character
considerations) we have that $N_2=L(0,2p-3)$, this results in a
filtration of $T(2p-2,p-2)$ with layers given by
$$T(p,0) | L(0,2p-3) | L(2p-2,p-2) | L(0,2p-3) | T(p,0).$$ 
%Suppose for
%a contradiction that this is the case. 
Now, setting $V =
L(p-2,p-2)$ to ease the notation, we have
\begin{equation}\label{eq:Hom-V-otimes}
\Hom(V \otimes L(p-2,1) , L(0,2p-3))
=\Hom(V, L(1,p-2)\otimes L(0,p-3)\otimes (E^\ast)^{[1]}).
\end{equation}
We shall show that the right hand side in the above is equal to zero
and thus arrive at a contradiction.

By Lemma \ref{Lemma3} of Section \ref{decomp}, the tensor product
$L(1,p-2)\otimes L(0,p-2)$ decomposes as a direct sum of
indecomposable tilting modules labelled by highest weights in alcoves
$1,2,3,4,4'$ along with a number of $p$-restricted simple modules.
(The proof of Lemma \ref{Lemma3} is independent of this subsection.)
We have constructed injections of all of these modules into modules of
the form $T(2(p-1)\rho+w_0\lambda)$ for $\lambda\in X_1$ and therefore
by Donkin's tilting tensor product theorem none of the summands of $
(L(1,p-2)\otimes L(0,p-2))\otimes L(0,1)^{[1]}$ have a $p$-restricted
simple module in their socle.  Therefore, the right hand side of
\eqref{eq:Hom-V-otimes} is zero, as claimed.  Therefore $N_1\neq
T(p,0)$.

This leaves us in the case that $N_1=0$, and it remains to deduce the
structure of $N_2$.  We know that $N_2$ is a self-dual quotient of
$P'(2)$ and therefore has head and socle isomorphic to $L(2)$.
Therefore $N_2$ appears in the list of modules in Figure
\ref{lotsapics}.  There are only two modules in Figure \ref{lotsapics}
with the correct character: namely the third and fourth. But only the
former module is contravariantly self-dual, so we conclude that it is
$N_2$. This shows that $N_2 = M(p-3,1)$.  Therefore $T(2p-2,p-2)$ does
possess the required third filtration.  One may conclude that the
desired filtration exists for any $T(\lambda)$ such that $\lambda \in
C_4$, by Jantzen's translation principle.  This completes the
proof.
\end{proof}
 
\begin{rmk}
  The proof of Lemma \ref{LEMER} shows that the module $M(\nu)$ for
  $\nu \in C_2$ sometimes appears as a direct summand of some tensor
  product of the form $L(\lambda)\otimes L(\mu)$ for $\lambda,\mu \in
  X_1$.  We shall see later that this is the case for all $\nu \in
  C_2$ and that these are the only non-simple, non-tilting
  indecomposable modules which can appear as a direct summand in such
  a tensor product.
\end{rmk}

\begin{prop}\label{T(4)}
  Let $\pi = \{1,2,3,3',4\}$. The basic algebra of the Schur algebra
  $S=S(\pi) = S(\le 4)$ is isomorphic to the path algebra of $Q$
  modulo the ideal generated by the following relations:
\begin{center}
$ c_i' c_i =d' d = 0,  c_i' b_i=d' a,
  d' a b_i' = d' a a' = 0,
 a a' a = 0,
a b_1'=-d c_1', a b_2'=d c_2', 
b_i b_i'=0, $ \\
$b_1 a'=   c_1 d',
 b_1 b_2'= c_1 c_2',     
   b_2 a'=-c_2 d',   
 b_2 b_1'=-c_2 c_1', 
 b_1' b_1+b_2' b_2=a' a,  
b_i' c_i=a' d, $
%$ c_i' c_i =d' d = 0,  c_i' b_i=d' a,
%  d' a b_i' = d' a a' = 0,
% a a' a = 0,
%a b_1'=-d c_1', a b_2'=d c_2', 
%b_i b_i'=0, $ \\
%$b_1 a'=   c_1 d',
% b_1 b_2'= c_1 c_2',     
%%b_2 b_2'=0,
%  b_2 a'=-c_2 d',   
% b_2 b_1'=-c_2 c_1', 
% b_1' b_1+b_2' b_2=a' a,  
%b_i' c_i=a' d, $
 \end{center}
where $i\neq j$.
\end{prop}

\begin{proof}
By Proposition \ref{prop:BH-reciprocity}, the projective indecomposable
modules for $S$ have the following $\Delta$-factors (going downwards)
$$
\begin{array}{ccc}
P(1) 		&& \Delta(1) | \Delta(2) | \Delta(4)\\
P(2) 		&& \Delta(2) | \Delta(3)\oplus \Delta(3') | \Delta(4) \\
P(3) 		&& \Delta(3) |  \Delta(4) \\
P(3') 		&& \Delta(3') | \Delta(4) \\
P(4) 		&&  \Delta(4) .
\end{array}
$$ 
The structure of $P(4)=\Delta(4)$ ensures that the relations $ c_i'
c_i= d' d =0, d' a b_i' = d' a a' = 0,$ $ d' a b_i' = d' a a' = 0$ all
hold.  We also see that $ c_i' b_i$, $d' a$ are all equal up to scalar
multiplication, and we may choose to take these coefficients to be
equal to 1.

The tilting module $T(6)$ is projective for the generalised Schur
algebra $S(\sigma)$ for $\sigma=\{1,2,3,3',4,6\}$, and its structure
has already been calculated.  We let $e_\pi$ denote the idempotent
corresponding to the subset of weights $\pi\subset \rho$ in which we
are interested. Applying the idempotent truncation map we see that
$e_\pi T(6) = P(1)$. Hence $P(1)$ has the following coefficient quiver
  $$
  \begin{minipage}{34mm}
   \def\objectstyle{\scriptstyle}\xymatrix@=6pt{
      &		&1 \ar@{-}[dl]	 \ar@{-}[dr]		&\\
      &2\ar@{-}[dl]\ar@{-}[dr]\ar@{-}[drrr]		&  	&4\ar@{-}[dr]\ar@{-}[dl]\ar@{-}[d] \\
1      &		&3		&1 &3' \\
      & 	&	&2\ar@{-}[ur]\ar@{-}[ul]\ar@{-}[u] 	   }
  \end{minipage}
  $$
Therefore
$a a' a = 0$, % as $\Ext^1(L(1),L(2))$ is 1 dimensional, and 
$a b_1'= \alpha_1d c_1', a b_2'= \alpha_2d c_2'$ where the 
$\alpha_i \in K$ are non-zero constants which are yet to be 
determined.

The coefficients we need to understand the structure of $P(3)$ are
$\beta_1,\beta_2,\beta_3 \in K$ where $b_1b_1'=\beta_1c_1c_1'$,
$b_1a'=\beta_2c_1d'$, $b_1b_2' = \beta_3 c_1c_2'$.  We have seen in
Lemma \ref{LEMER} that there exists a uniserial module of the form
$[3,4,3]$.  Therefore this module occurs as a quotient of the
projective $P(3)$, this implies that $\beta_1=0$.  By the structure of
$P(1)$, we know that there does not exist a uniserial module of the
form $[1,4,3]$ and therefore neither does there exist a module of the
form $[3,4,1]$; therefore $\beta_2 \in K$ is a non-zero constant that
is yet to be determined.

We now show that $\beta_3 \neq 0$.  The module $P(3)$ has two quotient
modules with socle $L(2)=\soc(T(4))$, namely $[3,2]$ and $P(2)$
itself, and $\dim_K \Hom(P(3),T(4)) = 2$ by Proposition \ref{donkin}.
Therefore $P(2)$ embeds into $T(4)$.  In what follows we shall
identify a simple composition factor of a projective module with the
path in the quiver which terminates at the given simple composition
factor.  The injection $f_1$ takes the simple head of $P(3)$ which is
labelled by the element $e_{3}$ to the simple composition factor
$L(3)$ in the second radical layer of $P(2)$ labelled by the path
$b_1'$.  Therefore $f_1$ takes the simple composition factor $L(2)$ in
the second radical layer of $P(3)$ labelled by $b_1$ to the simple
composition factor $L(2)$ in the third radical layer of $P(2)$
labelled by $b_1'b_1$.  The module $P(2)=T(4)$ is contravariantly
self-dual and so each copy of $L(2)$ in the third layer must extend at
least one of the $L(3)$ and $L(3')$ in the fourth layer.  We therefore
deduce that the simple composition factor $L(2)$ in the second Loewy
layer of $P(3)$ labelled by the path $b_1$ generates a submodule of
$P(3)$ with either an $L(3)$ or $L(3')$ as a composition factor.  We
have already seen that $L(3')$ is not a composition factor of this
module, therefore we conclude that $L(3)$ is a composition factor and
hence $\beta_3 \neq 0$.

Dual arguments hold for all but one of the above statements (allowing
us to make conclusions about $P(3')$ and the corresponding
coefficients $\gamma_1,\gamma_2,\gamma_3\in K$).  The statement which
has no dual comes from the fact that we have not constructed a
uniserial module $[3',4,3']$, i.e. we do not know if there does or
does not exist a uniserial module of the form $[3',4,3']$.  Therefore
the projective modules $P(3)$ and $P(3')$ have the following
coefficient quivers (where the extension corresponding to the dashed
line may or may not exist),
$$
  \begin{minipage}{34mm}
   \def\objectstyle{\scriptstyle}
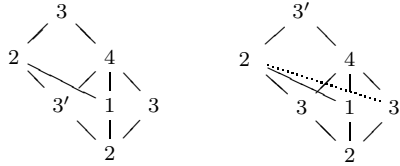
\xymatrix@=6pt{
      &		&3 \ar@{-}[dl]	 \ar@{-}[dr]		&\\
      &2 \ar@{-}[dr]\ar@{-}[drr]		&  	&4\ar@{-}[dr]\ar@{-}[dl]\ar@{-}[d] \\
       &		&3'		&1 &3  \\
      & 	&	&2\ar@{-}[ur]\ar@{-}[ul]\ar@{-}[u] 	   }
  \end{minipage}
  \quad
    \begin{minipage}{34mm}
   \def\objectstyle{\scriptstyle}\xymatrix@=6pt{
      &		&3' \ar@{-}[dl]	 \ar@{-}[dr]		&\\
      &2 \ar@{..}[drrr]\ \ar@{-}[dr]\ar@{-}[drr]		&  	&4\ar@{-}[dr]\ar@{-}[dl]\ar@{-}[d] \\
       &		&3 		&1 &3'  \\
      & 	&	&2\ar@{-}[ur]\ar@{-}[ul]\ar@{-}[u] 	   }
  \end{minipage}
$$ 
where the non-zero coefficients $\beta_2, \beta_3, \gamma_2, \gamma_3
$ are yet to be determined. The coefficient $\gamma_1$, corresponding
to the dashed line, will later be shown to be equal to zero.
 
We now consider the final projective module $P(2)=T(4)$.  A
$\nabla$-filtration of the module $T(4)$ has $\nabla(4)$ at the top,
and so the $b_ic_i'$, $a'd$ are non-zero and span a 1-dimensional
space.  We may pick the corresponding coefficients to be equal to 1.

Arguing as in the proof of Proposition \ref{T4}, we conclude that
$a'a=\zeta_1b_1'b_1+\zeta_2 b_2'b_2$.  This is the final path of
length two from vertex $2$ to itself. At this point, we make a
non-trivial choice by setting $\zeta_1=\zeta_2=1$, it is this choice
that determines the remaining coefficients.  In particular, the
quivers of all the self-dual proper quotient of $P(2)$ are given in
Lemma \ref{LEMER}.
 
Now consider the coefficients for $P(1)$.  There is no submodule of
$P(1)$ that is isomorphic to $M(2)$.  This can easily be seen as
$aa'a=0$.  Therefore the submodule of $P(1)$ generated by the simple
module $L(2)$ labelled by the path $a$ is isomorphic to the module
$P'(2)/\langle b_1'b_1+b_2'b_2 \rangle$ depicted in Figure
\ref{lotsapics}.  By our assumption on the coefficients above, this
implies that $ab_1'b_1=(-1)ab_2'b_2$.  This implies that
$\alpha_1=-\alpha_2$.

We now consider the submodule of $T(4)$ generated by the copy of
$L(2)$ in the third Loewy layer of $T(4)$ labelled by the path
$b_1'b_1-b_2'b_2$ (respectively $b_1'b_1+b_2'b_2$).  We have already
seen in Lemma \ref{LEMER} that $T(4)$ has a filtration
$[M(2),L(4),M(2)]$; we have chosen our coefficients so that the
submodule isomorphic to $M(2)$ is generated by the copy of $L(2)$
labelled by $b_1'b_1-b_2'b_2$.  This submodule  
has basis $\mathcal{B}=\{b_1'b_1-b_2'b_2,a'dc_1',a'dc_2',a'dd',a'dd'a\}$ with coefficient quiver
%%%%%%%%%%%%%%%%%%%%%%%%%%%%%%%%%%%%%%%%%%%%%%%%%%%%%%%%%%%%%%%%%%
%%%%% Maybe we don't need this big ugly diagram - SRD 
%%%%%%%%%%%%%%%%%%%%%%%%%%%%%%%%%%%%%%%%%%%%%%%%%%%%%%%%%%%%%%%%%%
%\[  
%\begin{minipage}{30mm}
%\def\objectstyle{\scriptstyle}\xymatrix@=6pt{
%&&	&	&	&	&&&&
%2
%\ar@{->}^{{\beta_3-\gamma_3}}[dddrrrrrr]
% \ar@{->}_{-\gamma_1}[dddllllll] 
% \ar@{->}^{\beta_2-\gamma_2}[ddd]		
% &\\ &\\&\\
%&&&3\ar@{->}[dddrrrrrr]		&&	&&&&1\ar@{->}[ddd]					&	&&			&&&3'
%\ar@{->}[dddllllll]%  \\ &\\&\\
%&&		&&&&&&&2		 	\\	
%}
%\end{minipage}
%\]
%%%%%%%%%%%%%%%%%%%%%%%%%%%%%%%%%%%%%%%%%%%%%%%%%%%%%%%%%%%%%%%%%%%
%%%%%%%%%%%%%%%%%%%%%%%%%%%%%%%%%%%%%%%%%%%%%%%%%%%%%%%%%%%%%%%%%%%
%given by the third diagram in Figure \ref{lotsapics}, in which %the
%  $N_{b_1',\mathcal{B}}(b_1'b_1-b_2'b_2, a'dc_1')=-\gamma_1$,
%%arrow from vertex $2$ to vertex $3$ is labelled by $-\gamma_1$, the
%  $N_{a',\mathcal{B}}(b_1'b_1-b_2'b_2, a'da')=\beta_2-\gamma_2$,
%%arrow from vertex $2$ to vertex $1$ is labelled by $\beta_2-\gamma_2$,
% $N_{b_2',\mathcal{B}}(b_1'b_1-b_2'b_2, a'dc_2')=\beta_3-\gamma_3$.
\begin{align*}
  N_{b_1',\mathcal{B}}(b_1'b_1-b_2'b_2, a'dc_1')&=-\gamma_1 \\
 N_{a',\mathcal{B}}(b_1'b_1-b_2'b_2, a'da')&=\beta_2-\gamma_2\\
 N_{b_2',\mathcal{B}}(b_1'b_1-b_2'b_2, a'dc_2')&=\beta_3-\gamma_3.
\end{align*}
We  deduce that $-\gamma_1=\beta_3-\gamma_3=\sigma$
and $\beta_2-\gamma_2= 2\sigma$ for some choice of $\sigma \in K$.
% We therefore deduce that $L(3)$, $L(1)$, and $L(3')$ all appear in
% this submodule and hence $ -\gamma_1, \beta_2 - \gamma_2,
% \beta_3-\gamma_3 \neq0$ (we have already seen that $\beta_1=0$).
% Moreover, we can choose all coefficients in a quiver for $M(2)$ (as
% in Figure \ref{lotsapics}).  Therefore we deduce that
% $-\gamma_1=\beta_3-\gamma_3=\sigma$ and $\beta_2-\gamma_2=2\sigma$
% for some non-zero constant $\sigma \in K$.  therefore the copy of
% $L(2)$ labelled .  The latter is studied via the third filtration in
% Lemma \ref{LEMER}.

We now study the submodule generated by the copy of $L(2)$ labelled by
$b_1'b_1+b_2'b_2$; to do this we study the image of a homomorphism
from $P(1)$ to $P(2)=T(4)$.  The homomorphism in which we are
interested is an injection of $P(1)/\langle aa' \rangle$ into
$T(4)$.  This takes the simple head of $P(1)$ labelled by the path
$e_1$ to the copy of $L(1)$ in the second Loewy layer of $P(2)$
labelled by $a'$.  It hence takes the simple composition factor $L(2)$
in the second radical layer of $P(1)$ labelled by the path $a$ to the
simple composition factor $L(2)$ in the third radical layer of $P(2)$
labelled by the path $a'a=b_1'b_1+b_2'b_2$.  From the structure of
$P(1)$ we deduce that the simple composition factor $L(2)$ in the
third radical layer of $P(1)$ labelled by the path $b_1'b_1+b_2'b_2$
generates a submodule as in the rightmost diagram in Figure
\ref{lotsapics}.  Therefore $\beta_2 +\gamma_2=0$ and and we may
choose the coefficients so that
%    $\beta_1-\gamma_1, \beta_3 -\gamma_3\neq 0$.  We may choose the
%    coefficients so that
$ \rho= ( \beta_1+\gamma_1)% \neq 0, 
= - (\beta_3 +\gamma_3)$ for some non-zero constant $\rho \in K$.

The unique solution to this set of relations is $\beta_1=0,
\beta_2=\beta_3=k$ and $\gamma_1=\gamma_2=-k$, $\gamma_3=0$ for $k=
(\sigma+\rho)/2$.  We may now fix $k=1$ and we are done.
\end{proof}

\begin{rmk}
  The authors have used the``Quivers and Path Algebras'' package
  \cite{QPA} for GAP \cite{GAP} to verify that the algebra defined by
  the quiver and relations of Proposition \ref{T(4)} has the correct
  dimension and that the projective modules have the correct dimension
  and Loewy series structure. 
\end{rmk}

We wish to consider the possible bases of coefficient quivers for the
tilting module $T(4)$.  The only choice to be made is which basis to
take for the 2-dimensional space $L(2)\oplus L(2)$ in the third Loewy
layer.
The most obvious choice of basis is given by $$\mathcal{B}=\{e_2,
b_1', a', b_2', b_1'b_1,a'd,b_2'b_2, a'dc_2', a'dd', a'dc_1',a'dd'a
\}$$ with respect to this basis the coefficient quiver is the leftmost
quiver in Figure \ref{2quivsfort4}.  The non-trivial coefficients in
this coefficient quiver are given by
$N_{b_1',\mathcal{B}}(b_2'b_2,a'dc_1') = -1$ and
$N_{a',\mathcal{B}}(b_2'b_2,a'dd')=-1$.  
Here we have taken
$N=\{b_1'b_1,b_2'b_2\}$ as the basis for the 
2-dimensional space $L(2)\oplus L(2)$ in the third Loewy layer.

 An alternative basis for the coefficient quiver is given by substituting 
$N'=\{b_1'b_1+b_2'b_2,b_1'b_1-b_2'b_2\}$ as the basis for the 
2-dimensional space $L(2)\oplus L(2)$ in the third Loewy layer.
 This  is depicted
in the rightmost diagram in Figure \ref{2quivsfort4}.

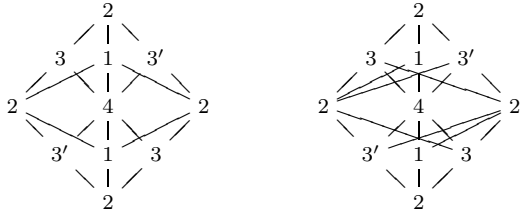
\begin{figure}[ht]
$$
  \begin{minipage}{34mm}
  \def\objectstyle{\scriptstyle}\xymatrix@=6pt{
&&	 					& 2 \ar@{-}[dr] \ar@{-}[dl] 			& 			&\\
&&3\ar@{-}[dl]   	& 1\ar@{-}[u] \ar@{-}[d]  \ar@{-}[dll] 	  \ar@{-}[drr] 	 					&3'   \ar@{-}[dr]	& \\
& 2  &			& 4  \ar@{-}[dr] \ar@{-}[dl] \ar@{-}[ur] \ar@{-}[ul] \ar@{-}[u]		& 	& 2\\
&&3'  \ar@{-}[ul]  			& 1\ar@{-}[ull]  \ar@{-}[u]  \ar@{-}[d]  \ar@{-}[urr]  	&3\ar@{-}[ur] \\
&&			& 2 \ar@{-}[ur] \ar@{-}[ul] 		&\\
 } \end{minipage} 
 \quad \quad
   \begin{minipage}{34mm}
  \def\objectstyle{\scriptstyle}\xymatrix@=6pt{
&&	 					& 2 \ar@{-}[dr] \ar@{-}[dl] 			& 			&\\
&&3\ar@{-}[dl]  \ar@{-}[drrr] 	& 1\ar@{-}[u] \ar@{-}[d]  \ar@{-}[dll] 	 					&3'  \ar@{-}[dlll] \ar@{-}[dr]	& \\
& 2 \ar@{-}[drrr] &			& 4  \ar@{-}[dr] \ar@{-}[dl] \ar@{-}[ur] \ar@{-}[ul] \ar@{-}[u]		& 	& 2\\
&&3'  \ar@{-}[ul] \ar@{-}[urrr]  			& 1\ar@{-}[u]  \ar@{-}[d]  \ar@{-}[urr]  	&3\ar@{-}[ur] \\
&&			& 2 \ar@{-}[ur] \ar@{-}[ul] 		&\\
 } \end{minipage} 
 $$
 \caption{Two coefficient quivers for the module $T(4)$. } \label{2quivsfort4}
\end{figure}

\subsection{} 
By \cite[Proposition 4.2]{Jensen}, the $p$-regular linkage class of
$E\otimes T(6|8)$ is the module $T(8)$.  Arguing as we did for the
second filtration in the proof of Lemma \ref{LEMER}, the head of
$T(8)$ can be seen to be $L(2) \oplus L(4)$.

We know that $\Delta(3)$ must extend $\Delta(2)$ by Proposition
\ref{donkin} and therefore these subquotients are correctly placed
within the diagram in Theorem \ref{thm:2}(b).  We know the character
of $P(4)$ by Proposition \ref{prop:BH-reciprocity}.  We now consider
$\Hom_{S(\le 8)}(P(4),T(8))$. Since $P(4)$ has a $\Delta$-filtration
and $T(8)$ has a $\nabla$-filtration we may apply Proposition
\ref{donkin} to see that $\dim_K\Hom_{S(\le 8)}(P(4),T(8)) = 4$.  This
allows us to place the other Weyl modules within the structure diagram
of $T(8)$ using Proposition \ref{BGG}.

\subsection{}\label{105}  
It follows from \ref{Jantzen-iso} that $T(5)$ and $T(7)$ are
projective-injective for suitable generalised Schur algebras.  The
projectivity of $T(9)$ can be seen by appealing to Proposition
\ref{prop:BH-reciprocity}.  Therefore Proposition \ref{BGG} gives the
Loewy structures of $T(5)$, $T(7)$ and $T(9)$ and proves that they are
rigid, as claimed in Theorem \ref{thm:2}.

Since the radical layers of these tilting modules are not
multiplicity-free, determining their full structure would be quite
complicated by these methods, so we do not pursue this further.

\section{Restricted tensor product decompositions: 
one or both factors tilting}\label{decomp}\noindent
We now turn to the proof of Theorem \ref{thm:1}, which is split over
this section and the next.  We need to show that each indecomposable
direct summand in a decomposition of a tensor product of two
$p$-restricted simple modules must have one of the the following
forms:
\medskip

  (a) a tilting module of highest weight $\lambda$ such that $\lambda
\le (2p-2)\rho$;

  (b) a simple module (which is not tilting) of highest weight in
  $C_2$; 

  (c) a module of the form $M(\lambda)$, for $\lambda \in C_2$.

\medskip\noindent
By highest weight considerations, it is easy to see that each of these
possibilities actually occurs in some restricted tensor product, hence
the above list provides a complete description of the isomorphism
classes of indecomposable summands of restricted tensor products for
$G=\SL_3$ (for $p\ge 5$). This will prove Theorem \ref{thm:1}.

We will see that there is an algorithm for the computation of the
multiplicities of the indecomposable direct summands of any
$p$-restricted tensor product.  In this section we freely switch
between alcove and highest weight notation depending on our
needs. Highest weights will usually be written using $\SL_3$ notation
but we shall sometimes find it convenient to use $\GL_3$ weight
notation in certain calculations; our conventions for such transitions
are laid out in \ref{ss:8:notation}.

Consider the set of $p$-restricted simple modules $L(\lambda)$ for
$\SL_3$.  If $\lambda \notin C_2$ then $L(\lambda)$ is tilting;
otherwise not.  These two cases therefore guide the calculation.  The
present section considers the indecomposable direct summands of
$L(\lambda) \otimes L(\mu)$ in case one or both of the factors is
tilting. The more difficult case, in which both factors are not
tilting, is considered in the next section.

\subsection{} 
For convenience, we work in the representation ring $\mathcal{R} =
\Rep_k(\SL_3)$, which is the quotient of the free abelian group on the
set $[L(\lambda)]$, as $\lambda$ varies over $X^+$, by the subgroup
generated by all expressions of the form $[M]-[M']-[M'']$ such that $0
\to M' \to M \to M'' \to 0$ is a short exact sequence of
finite-dimensional $G$-modules.

We have a ring homomorphism from $\mathcal{R}$ into the character ring
$\Z[X]^W$, for either case $X=X(T_{\GL_3})$ or $X=X(T_{\SL_3})$,
defined by sending $[M]$ for any module $M$ to its formal character
$\ch M \in \Z[X]^W$. This homomorphism is injective; i.e., $\ch M =
\ch N$ implies $[M]=[N]$ for any finite-dimensional modules $M,N$.
From highest weight considerations we know that any of the sets 
\[
\{[L(\lambda)] \colon \lambda \in X^+ \}, \quad\{ [\Delta(\lambda)]
\colon \lambda \in X^+ \}, \quad\{ [\nabla(\lambda)] \colon \lambda \in
X^+ \}, \quad \{ [T(\lambda)] \colon \lambda \in X^+ \}
\]
is a $\Z$-basis for $\mathcal{R}$. By highest weight theory, if $M$ is
a highest weight module of highest weight $\lambda$, then in
$\mathcal{R}$ we have $[M] = \sum_{\mu \le \lambda} m_\mu [L(\mu)]$.

\subsection{Both factors are tilting}\label{dddd}  
Since the tensor product of two tilting modules is tilting, any tensor
product $L(\lambda) \otimes L(\mu)$ of two $p$-restricted simples such
that $\lambda, \mu \notin C_2$ is tilting, and thus decomposes as a
direct sum of indecomposable tilting modules.  Furthermore, in this
case the modules $L(\lambda) = \Delta(\lambda)$, $L(\mu) =
\Delta(\mu)$ are also Weyl modules, therefore the (non-negative)
coefficients $c^\nu_{\lambda,\mu}$ in the expression
\begin{equation}\label{bt:1}
  [L(\lambda)\otimes L(\mu)] = [\Delta(\lambda)\otimes \Delta(\mu)] 
  = \sum_{\nu \in X^+} c^\nu_{\lambda,\mu} [\Delta(\nu)]
\end{equation}
are determined by the Littlewood--Richardson rule.  We know the
characters of the tilting modules by \cite{Jensen,Parker}; they also
appear implicitly in Section \ref{tilting}.  Thus we know the
coefficients in the expression
\begin{equation}\label{bt:2}
  [T(\nu)] = \sum_{\nu' \in X^+} t_{\nu,\nu'} [\Delta(\nu')].
\end{equation}
Note that $t_{\nu,\nu}=1$ since the highest weight space of any
indecomposable tilting module is known to have dimension 1, and
furthermore $t_{\nu,\nu'} = 0$ unless $\nu' \le \nu$.

This allows us to determine the multiplicities of the indecomposable
direct summands of $L(\lambda) \otimes L(\mu)$ by highest weight
theory, as follows: choose any $\nu$ which is maximal among the set of
all $\nu'$ such that $c^{\nu'}_{\lambda,\mu} \ne 0$ in the finite sum
in the right hand side of \eqref{bt:1}. Then $T(\nu)$ must occur
exactly $c^{\nu}_{\lambda,\mu}$ times as a direct summand of
$L(\lambda) \otimes L(\mu)$. Thus we subtract $c^{\nu}_{\lambda,\mu}
[T(\nu)] = \sum_{\nu' \in X^+} c^{\nu}_{\lambda,\mu} \, t_{\nu,\nu'}
[\Delta(\nu')]$ from the expression in the right hand side of
\eqref{bt:1}, and repeat the procedure on the difference. The process
terminates when the expression becomes zero, and termination after a
finite number of such steps is guaranteed.

\begin{example}
Let $p=5$ and consider $L(1,1) \otimes L(4,0)$. By applying the Pieri
rule for $\GL_3$ to the pair of partitions $((2,1,0))$ and $((4))$ and
restricting to $\SL_3$ (see \ref{ss:8:notation} for the notation and
conventions) we see that $[L(1,1) \otimes L(4,0)] = [\Delta(1,1)
  \otimes \Delta(4,0)] = [\Delta(5,1)] + [\Delta(3,2)] + [\Delta(4,0)]
+ [\Delta(2,1)]$.  From the known $\Delta$-filtration multiplicities
of the tilting modules, it follows that
\[
  L(1,1) \otimes L(4,0) \simeq T(5,1) \oplus T(4,0) \oplus T(2,1)
\]
as $[T(5,1) = [\Delta(5,1)]+[\Delta(3,2)]$, $[T(4,0)]=[\Delta(4,0)]$,
and $[T(2,1)]=[\Delta(2,1)]$.
\end{example}

\subsection{Only one factor is tilting}\label{c1c2}
We now consider the tensor product of $L(\lambda)$ for $\lambda \in
C_2$ with any $p$-restricted simple tilting module $L(\mu)$.  Thus,
$\mu$ is a $p$-restricted dominant weight belonging to $C_1$ or one of
the walls $\mathcal{F}_{1|2}$, $\mathcal{F}_{2|3}$,
$\mathcal{F}_{2|3'}$ of alcove $C_2$ (see Figure \ref{fig:1}), or
$L(\mu)=\St$ is the Steinberg module.

Since $L(\mu)$ is tilting, highest weight theory guarantees that it is
isomorphic to a direct summand of the tilting module $E^{\otimes
  \mu_1}\otimes (E^\ast)^{\otimes \mu_2}$, where $\mu = (\mu_1,
\mu_2)$.  Therefore we consider the tensor products $L(\lambda)\otimes
E$ and $L(\lambda) \otimes E^\ast$.

The following describes the Weyl filtration of a Weyl module tensored
by $E$ or $E^\ast$.

\begin{lem}
  Let $\lambda \in X^+$ be any dominant weight.  Recall that
  $\overline{\varepsilon}_j$ for $j=1,2,3$ are the weights of $E$. The
  weights of $E^\ast$ are $-\overline{\varepsilon}_j$ for
  $j=1,2,3$. In the representation ring $\mathcal{R}$ we have:

  (a)\quad $[\Delta(\lambda) \otimes E] = \sum_{j=1}^3
  [\Delta(\lambda+\overline{\varepsilon}_j)]$;

  (b)\quad $[\Delta(\lambda) \otimes E^\ast ]= \sum_{j=1}^3
  [\Delta(\lambda-\overline{\varepsilon}_j)]$

\noindent
with the stipulation that in the right hand side of each equality, we
omit any summand $[\Delta(\lambda \pm \overline{\varepsilon}_j)]$ of
non-dominant highest weight $\lambda \pm \overline{\varepsilon}_j$.
\end{lem}

\begin{proof}
  Let $\lambda=(\lambda_1, \lambda_2)$. We regard $\lambda$ as arising
  from the partition $((\lambda_1+\lambda_2, \lambda_2, 0))$ written
  as a $\GL_3$-weight with three components for the sake of
  emphasis. Tensoring by $E$ we apply the Pieri rule to get the
  decomposition with $\Delta$-factors of highest $\GL_3$-weight the
  partitions $((\lambda_1+\lambda_2+1, \lambda_2, 0))$,
  $((\lambda_1+\lambda_2, \lambda_2+1, 0))$, and
  $((\lambda_1+\lambda_2, \lambda_2, 1))$ except the last one does not
  occur if $\lambda_2=0$ and the second one doesn't appear if
  $\lambda_1=0$. Part (a) then follows by passing to $\SL_3$-weight
  notation.

  Given (a), we can apply it to decompose the tensor product
  $\Delta(-w_0(\lambda))\otimes E$, which gives
  \[
  [\Delta(-w_0(\lambda)) \otimes E] = \textstyle \sum _{j=1}^3
  [\Delta(-w_0(\lambda)+\overline{\varepsilon}_j)]
  \]
  with the stated stipulation.  Now formula (b) follows after applying
  the symmetry involution $-w_0$ to the weights of the above
  decomposition. Note that $-w_0(\overline{\varepsilon}_j) =
  -\overline{\varepsilon}_{4-j}$ for each $j=1,2,3$. Thus
  $-w_0(-w_0(\lambda)+\overline{\varepsilon}_j) = \lambda -
  \overline{\varepsilon}_{4-j}$ for each $j$, and (b) follows.
\end{proof}

\begin{lem}\label{2}
  For any $\lambda \in C_2$ we have the decompositions
  
  (a)\quad $L(\lambda) \otimes E \simeq 
  \begin{cases}
    L(\lambda+\overline{\varepsilon}_1) \oplus
    L(\lambda+\overline{\varepsilon}_2) \oplus
    L(\lambda+\overline{\varepsilon}_3), &\text{if }
    \lambda+\overline{\varepsilon}_3 \notin \mathcal{F}_{1|2}\\
    L(\lambda+\overline{\varepsilon}_1) \oplus
    L(\lambda+\overline{\varepsilon}_2)
    &\text{if } \lambda+\overline{\varepsilon}_3 \in
    \mathcal{F}_{1|2}.
  \end{cases}$

  (b)\quad $L(\lambda) \otimes E^\ast \simeq
\begin{cases}
    L(\lambda-\overline{\varepsilon}_1) \oplus
    L(\lambda-\overline{\varepsilon}_2) \oplus
    L(\lambda-\overline{\varepsilon}_3), &\text{if }
    \lambda-\overline{\varepsilon}_1 \notin \mathcal{F}_{1|2}\\
    L(\lambda-\overline{\varepsilon}_2) \oplus
    L(\lambda-\overline{\varepsilon}_3)
    &\text{if } \lambda-\overline{\varepsilon}_1 \in
    \mathcal{F}_{1|2}.
  \end{cases}$
\end{lem}

\begin{proof}
  For any $\lambda \in C_2$ we have $[L(\lambda)] = [\Delta(\lambda)]
  - [\Delta(s_{1|2}\cdot \lambda)]$. Thus we can compute $[L(\lambda)
  \otimes E]$ and $[L(\lambda)\otimes E^\ast]$ by the preceding
  lemma. The stated decompositions now follow from the linkage
  principle, which guarantees that the simple constituents of the
  tensor products cannot extend one another.
\end{proof}

We remark that the lemma holds for all primes $p \ge 3$, although we
will need it only in characteristics $p\ge 5$.

Examining the direct summands on the right hand side of either
decomposition (a) or (b) in Lemma \ref{2}, we observe (since $p\ge 5$)
that at most one of them can be tilting. In case (a) this happens if
and only if $\lambda+\overline{\varepsilon}_1 \in \mathcal{F}_{2|3}$
or $\lambda+\overline{\varepsilon}_2 \in \mathcal{F}_{2|3'}$, and in
case (b) it happens if and only if $\lambda-\overline{\varepsilon}_2
\in \mathcal{F}_{2|3}$ or $\lambda-\overline{\varepsilon}_3 \in
\mathcal{F}_{2|3'}$.  Furthermore, all the non-tilting direct summands
are of highest weight belonging to the alcove $C_2$.  So if we now
tensor by another $E$ or $E^\ast$ then, applying Lemma \ref{2} again
to the non-tilting summands, we see that the module can be written as a
direct sum of simple modules of highest weight in $C_2$, with one
additional summand which is either a tilting module or zero.

By induction on $\mu_1$ and $\mu_2$ it follows that the tensor product
$L(\lambda) \otimes E^{\mu_1} \otimes (E^\ast)^{\mu_2}$ can be
decomposed into a direct sum of simple modules of highest weight
belonging to $C_2$, modulo tilting summands. Since $L=L(\mu)$ is a
direct summand of $E^{\mu_1} \otimes (E^\ast)^{\mu_2}$, the module
$L(\lambda) \otimes L(\mu)$ is a direct summand of $L(\lambda) \otimes
E^{\mu_1} \otimes (E^\ast)^{\mu_2}$, and it follows that $L(\lambda)
\otimes L(\mu)$ is isomorphic to a direct sum of simple modules of
highest weight belonging to $C_2$, modulo tilting summands. Let us
record these observations.

\begin{lem}\label{Lemma3}
  For any $\lambda \in C_2$ and any $p$-restricted $\mu=(\mu_1,
  \mu_2)$ such that $L(\mu)=T(\mu)$ we have:

  (a) $L(\lambda)\otimes E^{\otimes \mu_1} \otimes (E^\ast)^{\mu_2}$
  is a direct sum of simple modules of highest weight in $C_2$ modulo
  tilting summands.

  (b) The same statement applies to $L(\lambda)\otimes L(\mu)$.
\end{lem}

Next we need to analyze the highest weights of the indecomposable
tilting summands that can occur. The main point is that they all have
highest weight some $\nu$ such that $\nu \notin C_1$. 

\begin{lem}\label{Lemma4}
  If $T(\nu)$ is an indecomposable tilting summand of the tensor
  product in (a) or (b) of the preceding lemma then $\nu \notin C_1$.
\end{lem}

\begin{proof}
  By induction it is enough to consider $T(\lambda)\otimes E$ and
  $T(\lambda)\otimes E^\ast$.  Considering the characters of the
  modules $E$ and $E^\ast$ it follows that if $\lambda \in X^+$ then
  \begin{equation}\label{eq:linkage-classes}
  T(\lambda) \otimes E =
  \sum_{j=1}^3 \mathrm{pr}_{\lambda+\overline{\varepsilon}_j}
  (T(\lambda) \otimes E); \quad T(\lambda) \otimes E^\ast =
  \sum_{j=1}^3 \mathrm{pr}_{\lambda-\overline{\varepsilon}_j}
  (T(\lambda) \otimes E^\ast) .
  \end{equation}
  It should be noted that the above sums are not always direct, as it
  can happen that two or more summands coincide. However, by
  definition of the functor $\mathrm{pr}_\mu$ it follows that two
  summands must be equal whenever they have non-trivial intersection.
 
  Now assume that $\lambda$ is a dominant weight not in alcove
  1. There are three cases to consider: (1) either $\lambda$ is a
  vertex (intersection point of two walls), (2) $\lambda$ is a weight
  on a wall which is not a vertex, or (3) $\lambda$ is a $p$-regular
  weight (and hence lies in the interior of an alcove). These cases
  are clearly mutually exclusive.

  If $\lambda$ is a vertex, then $T(\lambda)\otimes E$ and
  $T(\lambda)\otimes E^\ast$ have no indecomposable tilting summands
  of $p$-regular highest weight. Hence there can be no tilting
  summands of highest weight in alcove 1.

  If $\lambda$ is on a wall but is not a vertex, then Theorem 4.2 of
  \cite{Jensen} shows that the unique $p$-regular tilting summand of
  either tensor product $T(\lambda) \otimes E$ or $T(\lambda) \otimes
  E^\ast$ is indecomposable. Its highest weight is $\lambda+
  \overline{\varepsilon}_1$ and $\lambda- \overline{\varepsilon}_3$
  respectively, and thus cannot lie in alcove 1.

  Finally, suppose that $\lambda$ lies in the interior of its
  alcove. If $\lambda \pm \overline{\varepsilon}_j$ is $p$-regular
  then it lies in the interior of the same alcove.  By
  \eqref{eq:linkage-classes} and the translation principle, in this
  case $\mathrm{pr}_{\lambda + \overline{\varepsilon}_j}
  (T(\lambda)\otimes E) = T(\lambda + \overline{\varepsilon}_j)$ and
  $\mathrm{pr}_{\lambda - \overline{\varepsilon}_j} (T(\lambda)\otimes
  E^\ast) = T(\lambda - \overline{\varepsilon}_j)$ and neither highest
  weight $\lambda \pm \overline{\varepsilon}_j$ lies in alcove 1. 

  Note that whenever $\lambda \pm \overline{\varepsilon}_j$ is
  $p$-singular its corresponding linkage component in
  \eqref{eq:linkage-classes} cannot produce any tilting module of
  highest weight in alcove 1.
\end{proof}

\subsection{Decomposition algorithm, I}\label{algorithm:1}
Lemma \ref{Lemma4} implies that we can compute the multiplicities of
the indecomposable direct summands of $L(\lambda) \otimes L(\mu)$ by
the following algorithm:
\begin{enumerate}[label=(\alph*)]
\item Express $[L(\lambda) \otimes L(\mu)]$ in terms of the
  $[\Delta(\nu)]$-basis. This can be done using two applications of
  the Littlewood--Richardson rule (as in the proof of Lemma \ref{2})
  as follows
  \begin{equation*} 
  [L(\lambda)\otimes L(\mu)] = [\Delta(\lambda)\otimes \Delta(\mu)] -
  [\Delta(s_{1|2}\cdot \lambda)\otimes \Delta(\mu)] 
  = \sum_{\nu \in X^+} ( c^\nu_{\lambda,\mu}  
  - c^\nu_{s_{1|2}\cdot \lambda,\mu} ) [\Delta(\nu)].
  \end{equation*}
 \item 
  Express $[L(\lambda) \otimes L(\mu)]$ in terms of the
  $[L(\nu)]$-basis; i.e., compute the composition factor
  multiplicities in both filtrations and their difference. 
   This produces an expression of the form 
  \[
  [L(\lambda) \otimes L(\mu)] = \textstyle \sum_\nu d^{\lambda,\mu}_\nu [L(\nu)]
  \]
  in which each $d^{\lambda,\mu}_\nu \ge 0$. 

\item If $\nu \notin C_1 \cup C_2$ is maximal such that
  $d^{\lambda,\mu}_\nu >0$, then subtract $d^{\lambda,\mu}_\nu
  [T(\nu)]$. Repeat on the difference, until there do not exist any
  $\nu \notin C_1 \cup C_2$ appearing in the expression.

\item At this point, only terms of the form $[L(\nu)]$ for $\nu \in
  C_1 \cup C_2$ will remain. So we are dealing with an expression of
  the form $\sum_{\nu \in C_1 \cup C_2} b_\nu [L(\nu)]$.  Each term of
  the form $b_\nu [L(\nu)]$ for $\nu \in C_1$ must, by Lemma
  \ref{Lemma4}, be a composition factor of some tilting module of
  highest weight in $C_2$. Since $[T(2)] = [L(2)] + 2[L(1)]$ it
  follows that each $b_\nu$ for $\nu \in C_1$ is even, and $b_{s_{1|2}
    \cdot \nu} \ge \frac{b_\nu}{2}$. Subtract $\frac{b_\nu}{2}
  [T(s_{1|2} \cdot \nu)]$ from the expression for each such $\nu \in
  C_1$. The remaining expression is a linear combination of various
  $[L(\nu)]$ for $\nu \in C_2$, and each of these simples appears as a
  summand of the decomposition according to its multiplicity.
\end{enumerate}

To summarise step (d): we have shown that for $\nu \in C_1$
representing a given linkage class, the multiplicity of
$T(s_{1|2}\cdot \nu)$ in $L(\lambda) \otimes L(\mu)$ is
$\frac{b_\nu}{2}$ and the multiplicity of $L( \nu)$ is $b_{s_{1|2}
  \cdot \nu} - \frac{b_\nu}{2}$.  In other words, the multiplicities
of $T(s_{1|2}\cdot \nu)$ and $L( \nu)$ in the tensor product are given
by the matrix product
\[
\left(\begin{array}{rl}1/{2} & 0 \\ - {1}/{2}  & 1\end{array}\right)
\left(\begin{array}{c}b_\nu \\ b_{s_{1|2}\cdot \nu} \end{array}\right).
\]
We note that the above matrix is simply the inverse of 
\[
\left(\begin{array}{cc}2 & 0 \\1 & 1\end{array}\right),
\]
the change of basis matrix expressing the basis $\{[T(2)],[L(2)]\}$ in
terms of the basis $\{ [L(1)],$ $[L(2)] \}$.  A similar issue arises
in Section \ref{OHMYGOD}.
 
\begin{eg} 
Suppose the characteristic is $p=5$.  Take the tensor product $L(2,2)
\otimes L(1,1)$.  The Littlewood--Richardson rule gives the character
\begin{center}
$[L(3,3)] + [L(1,4 )]+[L(4,1)]+[L(2,2)],$
\end{center}
and this can be seen to give a direct sum decomposition by the linkage
principle.
\end{eg}

\section{Restricted tensor product decompositions: 
neither factor is tilting}\label{sec:c2c2}\noindent 
It remains to describe the indecomposable direct summands of
restricted tensor products $L(\lambda) \otimes L(\mu)$ in which
neither factor is tilting (i.e., both $\lambda, \mu \in C_2$).
Analysis of this remaining case will complete the proof of Theorem
\ref{thm:1}.

\subsection{}\label{c2c2}
To do this we will need to consider the set of weights in $C_2$ along
the strip
\[
  \{\nu \in C_2: \bil{(\alpha_1+\alpha_2)^\vee}{\nu+\rho}=p-1\}
\] 
an example of which is pictured in Figure \ref{trace}. We refer to
this set as the set of \emph{minimal} weights in $C_2$.  We first
consider a tensor product of the form $L(\sigma)\otimes L(\tau)$,
where $\sigma$ is the unique minimal weight of the form
$\lambda-s\overline{\varepsilon}_1$ and $\tau$ the unique minimal
weight of the form $\mu-t\overline{\varepsilon}_1$.  It follows from
Lemma \ref{Lemma3} that $L(\lambda)$ and $L(\mu)$ are direct summands
of $E^{\otimes s} \otimes L(\sigma)$ and $E^{\otimes t} \otimes
L(\tau)$, respectively.  Therefore $L(\lambda)\otimes L(\mu)$ is a
direct summand of
\[
  E^{\otimes {(s+t)}} \otimes L(\sigma) \otimes L(\tau).
\]
This will allow us to prove results for the decomposition of arbitrary
tensor products of the form $L(\lambda) \otimes L(\mu)$ for
$\lambda,\mu \in C_2$ by induction.

\begin{eg}
For example, take $p=7$, $\lambda=(3,5)$ and $\mu=(5,4)$.  In this
case $s=2$, $t=3$ and so $\sigma=(1,5)$ and $\tau=(2,4)$.  Therefore,
$L(\lambda)\otimes L(\mu)$ is a direct summand of
\[
  E^{\otimes 5} \otimes L(1,5)\otimes L(2,4).  
\]
One can see this in Figure \ref{trace}.
\begin{figure}[h]
\begin{center}
\begin{tikzpicture}[scale=0.7]
  \path (0,0) coordinate (origin);
  \path (-60:7cm) coordinate (A1);
  \path (240:7cm) coordinate (A2);
    \foreach \i in {1,...,7}
  {
    \path (origin)++(-60:1*\i cm)  coordinate (a\i);
    \path (origin)++(240:1*\i cm)  coordinate (b\i);
       \path (A2)++(0:1*\i cm)  coordinate (c\i);
       \path (A1)++(180:1*\i cm)  coordinate (d\i);
     \draw[thin,gray] (a\i) -- (c\i) (a\i) -- (b\i) (b\i)--(d\i); 
  }
  \draw[thick,black] (A1) -- (A2) (origin) -- (A1) (origin) -- (A2);
                                       \draw (d1)++(120:1cm) node {\tiny $ (5,1)$};    \draw (d2)++(120:1cm) node {\tiny $ (4,2)$};
                                       \draw (d3)++(120:1cm) node {\tiny $ (3,3)$};    
                                       \draw (d4)++(120:1cm) node {\tiny $ (2,4)$};
                                       \draw (d5)++(120:1cm) node {\tiny $ (1,5)$};    
                                         \draw (d1)++(115:4cm) node {\small  $\mathbf{ \mu}$};    \draw (d3)++(114:3.07cm) node {\small $ \mathbf{\lambda}$};
\end{tikzpicture}
\end{center}
\caption{The closure of the alcove ${C}_2$, for $p=7$.  The minimal
  weights in $C_2$, $\lambda=(3,5)$ and $\mu=(5,4)$ are at the
  labelled points.}
\label{trace}
\end{figure}
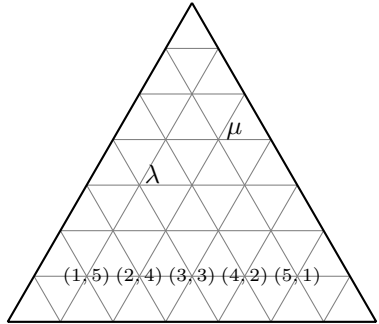 
\end{eg}

\subsection{}\label{bound}  
We will eventually show that $L(\lambda)\otimes L(\mu)$, for
$\lambda,\mu \in C_2$, decomposes as a direct sum of tilting modules
and modules of the form $M(\nu)$ for $\nu \in C_2$.  Sections
\ref{bound} through \ref{844} focus on minimal tensor products.
Section \ref{MtoM} will deal with the general case.

Let $\sigma, \tau$ be any two minimal weights.  There exist $0\le a,b
\le p-3$, such that $\sigma=(p-2-a,a+1)$ and $\tau=(p-2-b,b+1)$.  We
let
\[
\sigma'=
\begin{cases}
s_{2|3}\cdot \sigma \text{ if } a+b \le p-3 \\
s_{2|3'}\cdot \sigma \text{ if } a+b  > p-3
\end{cases}
\text{ and } 
\tau'=
\begin{cases}
s_{2|3}\cdot \tau \text{ if } a+b \le p-3 \\
s_{2|3'}\cdot \tau \text{ if } a+b  > p-3.
\end{cases}
\]
Note that in the case that $0\le a+b \le p-3$, $\sigma'$ and
$\tau'$ are of the form $(p+a,0)$ and $(p+b,0)$, respectively (the
other case is obtained by symmetry).  It is clear from Section
$\ref{sec:Weyl}$ that we get injections
$L(\sigma)\hookrightarrow\Delta(\sigma')\hookrightarrow T(\sigma')$.
We will consider the images of the injective homomorphisms
\begin{equation}\label{inject}
  L(\sigma) \otimes L (\tau) \hookrightarrow \Delta(\sigma')\otimes
  \Delta (\tau') \hookrightarrow T(\sigma')\otimes T (\tau') .
\end{equation} 
%We shall do this for  $0\le a + b \le p-3 $.     
%In the case that  $p-3\le a + b \le 2(p-3)$, one can argue by symmetry.   

%\subsubsection{}\label{TT}
The tensor product $T(\sigma')\otimes T (\tau')$ has a
$\Delta$-filtration and decomposes as a direct sum of tilting modules.
We can bound the highest weights of the $\Delta$-modules in such a
filtration as illustrated in Figure \ref{bear}.
In the $0\le a+b \le p-3$ case,
\begin{align*}
  [T(\sigma')] &= [\Delta(p+a,0)] +[ \Delta(p-2-a,a+1)] \\
  [T(\tau')] &= [\Delta(p+b,0)] +[ \Delta(p-2-b,b+1)]
\end{align*}
and so the highest weights that appear are bounded by the
${\alpha_2}$-string through $(p+a,0)+(p+b,0)=(2p+a+b,0) \in C_6$.  The
other case, $p-3< a+b \le 2(p-3)$, is similar.

\begin{figure}[h]
\begin{center}
\begin{tikzpicture}[scale=1]
   \path (0,0) coordinate (origin);
  \path (60:1cm) coordinate (A1);
  \path (60:2cm) coordinate (A2);
  \path (60:3cm) coordinate (A3);
   \path (120:1cm) coordinate (B1);
  \path (120:2cm) coordinate (B2);
  \path (120:3cm) coordinate (B3);
   \path (A1) ++(120:2cm) coordinate (C1);
  \path (A2) ++(120:1cm) coordinate (C2);
   \draw[thick] (origin) -- (A3) (origin) -- (B3) (A3) -- (B3)
        (A1) -- (C1) (A2) -- (C2) 
        (B1) -- (C2) (B2) -- (C1)  
        (B1) -- (A1) (B2) -- (A2)  ;
          \path (origin) ++(120:0.0) coordinate (C3);
            \draw (origin) ++(0,0.6) node {\small$\mathbf{1}$};
             \draw (origin) ++(0,1.1) node {\small$\mathbf{2}$};
  \draw (A1) ++(0,0.6) node {\small$\mathbf{3}$};
  \draw (A1) ++(0,1.1) node {\small$\mathbf{4}$};
  \draw (B1) ++(0,0.6) node {\small$\mathbf{3}'$};
  \draw (B1) ++(0,1.1) node {\small$\mathbf{4}'$};
  \draw (A2) ++(0,0.6) node {\small$\mathbf{6}$};
   \draw (B2) ++(0,0.6) node {\small$\mathbf{6}'$};
    \draw (A1) ++(120:1cm) ++(0,0.6) node {\small$\mathbf{5}$};
  \path (C3) ++(60:2.6) coordinate (C2);
    \path (C2) ++(-30:0.3cm) coordinate (D2);
        \path (C2) ++(210:2.4cm) coordinate (E2);
  \draw[very thick,gray]  (C2) -- (E2); %(C2) -- (D2) 
 \end{tikzpicture}\quad\begin{tikzpicture}[scale=1]
   \path (0,0) coordinate (origin);
  \path (60:1cm) coordinate (A1);
  \path (60:2cm) coordinate (A2);
  \path (60:3cm) coordinate (A3);
   \path (120:1cm) coordinate (B1);
  \path (120:2cm) coordinate (B2);
  \path (120:3cm) coordinate (B3);
    \draw (origin) ++(0,0.6) node {\small$\mathbf{1}$};
             \draw (origin) ++(0,1.1) node {\small$\mathbf{2}$};
  \draw (A1) ++(0,0.6) node {\small$\mathbf{3}$};
  \draw (A1) ++(0,1.1) node {\small$\mathbf{4}$};
  \draw (B1) ++(0,0.6) node {\small$\mathbf{3}'$};
  \draw (B1) ++(0,1.1) node {\small$\mathbf{4}'$};
  \draw (A2) ++(0,0.6) node {\small$\mathbf{6}$};
   \draw (B2) ++(0,0.6) node {\small$\mathbf{6}'$};
       \draw (A1) ++(120:1cm) ++(0,0.6) node {\small$\mathbf{5}$};
   \path (A1) ++(120:2cm) coordinate (C1);
  \path (A2) ++(120:1cm) coordinate (C2);
   \draw[thick] (origin) -- (A3) (origin) -- (B3) (A3) -- (B3)
        (A1) -- (C1) (A2) -- (C2) 
        (B1) -- (C2) (B2) -- (C1)  
        (B1) -- (A1) (B2) -- (A2)  ;
%%%%%%%%%%%%%%%%%%%%%%%%%%%%%
          \path (origin) ++(60:0.0) coordinate (C3);
   \path (C3) ++(120:2.6) coordinate (C2);
    \path (C2) ++(210:0.3cm) coordinate (D2);
        \path (C2) ++(-30:2.4cm) coordinate (E2);
        %%%%%%%%%%%%%%%%%%%%%%%%
  \draw[very thick,gray]  (C2) -- (E2); % (C2) -- (D2)
 \end{tikzpicture}
\end{center}
\caption{A typical example of the highest weights of $\Delta$-modules
  in a filtration of $T(\sigma')\otimes T(\tau')$ for $0\le a+b\le
  p-3$ and $p-3<a+b\le 2(p-3)$ respectively.  The lines cutting across
  the alcoves are the ${\alpha_j}$-strings through $(2p+a+b)\varpi_i$
  which bound the weights above (for $i\neq j$).  }
 \label{bear}
\end{figure}
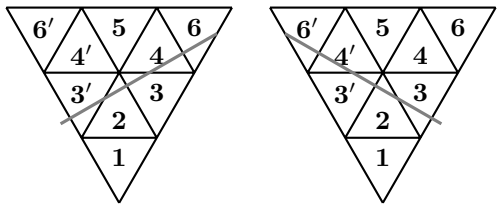

Having bounded the character of the tensor product, we may now
conclude that $T(\sigma')\otimes T(\tau')$ is a direct sum of tilting
modules of highest weights in $C_1,C_2,C_3,C_3',C_4,C_6,
\mathcal{F}_{2|3'}$, $\mathcal{F}_{2|3}, \mathcal{F}_{3|4},$
$\mathcal{F}_{4|6}$ when $0\le a+b \le p-3$ (and those obtained by
symmetry in the case $p-3<a+b \le 2(p-3)$).

By (\ref{inject}), $L(\sigma) \otimes L(\tau)$ appears as a submodule
of such a tilting module.  The simple modules are themselves
contravariantly self-dual and therefore the tensor product $L(\sigma)
\otimes L(\tau)$ is also contravariantly self-dual.  Finally, the
weights, $\lambda$, in $L(\sigma) \otimes L(\tau)$ satisfy the
inequality $\langle (\alpha_1+\alpha_2)^\vee , \lambda+\rho\rangle \le
2p-2$; therefore the simple modules in the tensor product $L(\sigma)
\otimes L(\tau)$ have highest weights in $C_1,C_2,C_3,C_3',
\mathcal{F}_{2|3}, \mathcal{F}_{2|3'}, \mathcal{F}_{3|4}$ (or the set
obtained by symmetry).

%We now seek to describe such modules.
We shall focus on the $0\le a+b\le p-3$ case, as the other case is
obtained by symmetry.

\subsection{} \label{table}
We now focus on the modules $\Delta( \sigma')\otimes \Delta(\tau')$,
in order to study $L( \sigma')\otimes L(\tau')$.  In the
representation ring, the decomposition of $\Delta( \sigma')\otimes
\Delta(\tau')$ is given by the Littlewood--Richardson rule as follows
\begin{equation} \label{jormjorm}
  [\Delta(p+a,0) \otimes \Delta(p+b,0)]= \sum_{0\le j \le p+ b}
  [\Delta(2p+a+b-2j ,j)].
\end{equation} 
All these weights appear along the ${\alpha_2}$-string through
$(2p+a+b,0)$, as illustrated in Figure \ref{bear}.
The projection of $\Delta( \sigma')\otimes \Delta(\tau')$ onto any
linkage class has a $\Delta$-filtration.  When we project onto a
linkage class there are six distinct cases which can occur.  These are
summarised in the table below.  We label each case by the highest
weight in the linkage class.
\[
\begin{array}{c|l|l}
\text{character } &   \text{ highest  weight  } &\text{ condition}	\\
\hline
{[\Delta(2|3)]} &(p-1,\tfrac{1}{2}(p+1+a+b)) &a+b \text{ is even} \\
    {[\Delta(4|6)]} 		&(2p-1, \frac{1}{2}(a+b+1) ) &a+b \text{ is odd}				\\
    {[\Delta(3|4)]} +     {[\Delta(2|3')]} & (2p-4-a-b,2+a+b)      &\text{ for all } a,b    \\
    {[\Delta(3)]} +     {[\Delta(2)]}  &(2p+a+b-2j,j) &  \frac{1}{2}(p  +  1  +  a  +  b)\le j<2 + a + b \\
   \quad     {[\Delta(6)} +     {[\Delta(4)]} &(2p+a+b-2j,j) & a<2j < a+b+1 \\
        {[\Delta(6)}] +     {[\Delta(4)]} +     {[\Delta(3')]}   & (2p+a+b-2j,j) &2j\le a
 \end{array}
\]

\subsection{}\label{abc}
For each possible linkage class of $\Delta (\sigma') \otimes \Delta
(\tau')$ in the table in \ref{table}, we wish to calculate the image
\begin{equation}\label{importantembedding}
  L(\sigma)\otimes L(\tau) \hookrightarrow \Delta (\sigma') \otimes
  \Delta (\tau').
\end{equation} 
In this section, we shall deal with the first five cases.  We shall
see that in these five cases, all possible summands are tilting.  By
Section \ref{sec:Weyl}, the character of the image can easily be seen
to be given by
\begin{align*}
  [L(\sigma) \otimes L(\tau)] &= [\Delta(\sigma') \otimes
    \Delta(\tau')] - [L(\sigma')\otimes L(\tau)] - [L(\sigma)\otimes
    L(\tau')] \\ &= [\Delta(\sigma')\otimes \Delta(\tau')] -
  [L(1,0)^{[1]}\otimes L(a,0)\otimes L(\tau)] - [L(\sigma)\otimes
    L(1,0)^{[1]}\otimes L(b,0)].
\end{align*}
Note that, by the Steinberg tensor product theorem, none of the
subtracted terms have $p$-restricted composition factors.  Therefore,
$L(\sigma) \otimes L(\tau)$ is a submodule of $\Delta(\sigma') \otimes
\Delta(\tau')$ which
\begin{enumerate}[label={\rm(\alph*)},leftmargin=*] %,itemsep=0.5em]
\item is contravariantly self-dual; \item has simple composition
  factors whose highest weights satisfy the inequality $$\langle
  (\alpha_1+\alpha_2)^\vee , \lambda+\rho\rangle \le 2p-2;$$ \item
  satisfies $[\Delta(\sigma')\otimes\Delta(\tau'): L(\nu)]=[L(\sigma)
    \otimes L(\tau):L(\nu)]$ for any $\nu \in X_1$.
\end{enumerate} 

We proceed case by case through the table in Section \ref{table}.  In
the first case described in the table, we see for $a+b$ even, that
$\Delta(2|3)=L(2|3)$ appears as a direct summand of
$\Delta(\sigma')\otimes \Delta(\tau')$.  By condition (c) this implies
that $L(2|3)$ appears as a direct summand of $L(\sigma) \otimes
L(\tau)$.
 
In the second case, $\Delta(4|6)=[L(4|6),L(1|2)]$ appears as a summand
of $\Delta(\sigma')\otimes \Delta(\tau')$.  By conditions (b) and (c)
this implies that $L(1|2)$ appears as a direct summand of $L(\sigma)
\otimes L(\tau)$.

In the third case, we see that $N$ appears as a summand of
$\Delta(\sigma')\otimes \Delta(\tau')$, where $N$ is an extension of
the form
\[
  0 \to \Delta(3|4) \to N \to \Delta(2|3') \to 0.
\]
The issue is whether or not this splits.  By (a) and (c), either the extension
is split and $L(2|3')\oplus L(2|3')$ is a direct summand of
$L(\sigma') \otimes L(\tau')$; or the sequence is the unique non-split
extension, $N\cong T(3|4)$, and $T(3|4)$ is a direct summand of
$L(\sigma') \otimes L(\tau')$.  In either case, the result is a sum of
tilting modules (note that $L(2|3')=T(2|3')$).

In the fourth case, we have a direct summand, $N$, of $\Delta(\sigma')
\otimes \Delta(\tau')$ where $N$ is an extension of the form
\[
  0 \to \Delta(3) \to N \to \Delta(2) \to 0.
\]
By Section \ref{sec:Weyl}, we know that $L(1)$ appears exactly once as
a composition factor of $N$ and that $\Delta(2)$ is a non-split
extension of $L(2)$ by $L(1)$, which, by (c) is preserved under the
embedding of (\ref{importantembedding}).  Therefore if $N$ is a split
extension, then there exists no submodule of $N$ satisfying properties
(a) and (c), which is a contradiction.  Hence $N$ is the unique
non-split extension and isomorphic to $T(3)$.  Now, notice that the
only submodule of $T(3)$ satisfying properties (a), (b) and (c) is
$T(3)$ itself.

In the fifth case, we have a direct summand, $N$, of $\Delta(\sigma')
\otimes \Delta(\tau')$ where $N$ is an extension of the form
\[
  0 \to \Delta(6) \to N \to \Delta(4) \to 0.
\]
Note that $L(2)$ appears with multiplicity one in $N$ and it extends a
$p$-restricted simple module.  One can now argue as above by noting
that if this extension is split we arrive at a contradiction.
Therefore, $N$ is the unique non-split extension. The only possible
contravariantly self-dual submodule of $N$ satisfying (c) is $T(2)$.

\subsection{}\label{jamjam} 
We now deal with the final case in the table in Section \ref{table}.
Our aim in this section is to show that in the final case, the summand
of $\Delta(\sigma') \otimes \Delta(\tau')$ is isomorphic to $\Delta(6)
\oplus N$ where $N$ is a non-split extension of the form
\[
  0 \to \Delta(4) \to N \to \Delta(3) \to 0,
\]
and that $L(1)\oplus M(2)$ is a direct summand of $L(\sigma) \otimes
L(\tau)$.  The character of the tensor product $T(\sigma')\otimes
T(\tau')$ is given by
\begin{equation} \label{tiltingcharacter}
  [T(\sigma')\otimes T(\tau')] = [\Delta(\sigma')\otimes
    \Delta(\tau')] + [\Delta(\sigma)\otimes \Delta(\tau')] \\ +
  [\Delta(\sigma')\otimes \Delta(\tau)] + [\Delta(\sigma)\otimes
    \Delta(\tau)].
\end{equation}

  We shall focus on linkage classes appearing along the $\alpha_2$-root string through $(2p+a+b,0)$.
Calculation of any of the tensor product decompositions along this
${\alpha_{2}}$-root string is easily done using the
Littlewood--Richardson rule (it is identical to the $\SL_2$ case).
Let $(c_1,c_2)= (a_1+b_1-2i,a_2+b_2+i)$ for $i\le \tfrac{1}{2}
(a_1+b_1)$.  Then
\begin{align*}
  [ \Delta(a_1,a_2) \otimes \Delta( b_1,b_2 ) : \Delta (c_1, c_2) ] =
  \begin{cases}
   1 & \text{for $i\le \min\{a_1,b_1\}$}, \\ 0 & \text{otherwise}.
  \end{cases}
\end{align*}  
Applying this to the four terms in the right-hand side of equation
(\ref{tiltingcharacter}) %(i.e., along the ${\alpha_2}$-strings in
%Figure \ref{dannyboy}), 
and projecting onto the  linkage class   with highest
weight $(2p+a+b-2j,j), \text{ for }2j\le a$, we get
$[T(\sigma')\otimes T(\tau'): \Delta(6)] = 1$, $[T(\sigma')\otimes
  T(\tau'): \Delta(4)] = 3$ and $[T(\sigma')\otimes T(\tau'):
  \Delta(3')] = 2$.

%\textcolor{red}{We find the previous paragraph hard to follow. The
%  figure does not seem to be of any use (to us) but that may be
%  because we missed the point somehow. Can this be clarified?}

Highest weight theory tells us that the  linkage class   of $T(\sigma')\otimes
T(\tau')$ in which we are interested therefore has $T(6)\oplus 2T(4)$
as a direct summand (using the method highlighted in Section
\ref{dddd}).  From the characters of $T(6)$ and $T(4)$, we deduce that
each of the two copies of $\Delta(3')$ which occur in a Weyl
filtration of $T(\sigma')\otimes T(\tau')$ must occur in one of the
summands isomorphic to $T(4)$. Therefore $T(3')$ is not a direct
summand of the  linkage class  .

We now turn our attention to the submodule $\Delta(\sigma')\otimes
\Delta(\tau')\hookrightarrow T(\sigma')\otimes T(\tau')$.  The
character of the projection of $\Delta(\sigma')\otimes \Delta(\tau')$
onto the  linkage class   in which we are interested is
$[\Delta(6)]+[\Delta(4)]+[\Delta (3')]$.  The corresponding module
appears as a submodule of $T(\sigma')\otimes T(\tau')$ and so is
isomorphic to $\Delta(6)\oplus N$, where $N$ is the unique non-split
extension
\[
  0 \to \Delta(4) \to N \to \Delta(3') \to 0,
\]
as we have shown that any $\Delta(3')$ must appear in a $T(4)$.  
%Note that $N$ is isomorphic to $P(3)$ for $S(\le4)$.

Finally, $L(\sigma)\otimes L(\tau)$ has character
\begin{align*}
  [L(\alpha)\otimes L(\beta)] =& [\Delta(p+a,0)\otimes\Delta(p+b,0)]-
  [L(p+a,0)\otimes L(p+b,0)] \\ &- [L(p+a,0)\otimes L(p-2-b,1+b)] -
  [L(p+b,0) \otimes L(p-2-a,1+a)].
\end{align*}
By the Steinberg tensor product theorem, neither of the latter two
terms contains an $L(3)$ or an $L(3')$.  Considering the second term,
we have that
\begin{align*}
L(p+a,0) \otimes L(p+b,0) &\cong (L(1,0) \otimes L(1,0))^{[1]} \otimes
(L(a,0)\otimes L(b,0))\\ &\cong (L(2,0) \oplus L(0,1))^{[1]} \otimes
(L(a,0)\otimes L(b,0)).
\end{align*}
By the Steinberg tensor product theorem, $L(2,0)^{[1]} \otimes
(L(a,0)\otimes L(b,0))$ does not contain an $L(3)$ or an $L(3')$.
However, $L(3')$ does appear in $L(0,1)^{[1]} \otimes (L(a,0)\otimes
L(b,0))$ with multiplicity equal to 1.

%To see this note that $a\varpi_1, b\varpi_1 \in C_1$ and
%so $$L(a\varpi_1)\otimes L(b\varpi_1)\cong \oplus_{j \le b}
%L((a+b-2j)\varpi_1+ j\varpi_2)$$ (using the Littlewood--Richardson
%rule and the fact that all the Weyl modules in $C_1$ are simple).
%Therefore, $$ L(\varpi_2)^{[1]} \otimes (L(a\varpi_1)\otimes
%L(b\varpi_1))\cong\oplus_{j \le b} L(
%j\varpi_1+(p+a+b-2j)\varpi_2).$$ It is easy to see that these weights
%lie across the $\alpha_2$-string through $(2p+a+b)\varpi_1$ and
%therefore deduce that $L(3')$ appears with multiplicity 1.

To summarise, we now know that $L(\sigma)\otimes L(\tau)$ satisfies
properties (a), (b) and (c) of Section \ref{abc} and that
$[L(\sigma)\otimes L(\tau): L(3')]=[\Delta(\sigma')\otimes
  \Delta(\tau'): L(3')]-1$ and $[L(\sigma)\otimes L(\tau):
  L(3)]=[\Delta(\sigma')\otimes \Delta(\tau'): L(3)]$.  There is a
unique possible submodule which obeys all these properties, given as
follows
\[
  L(1)\oplus M(2) \hookrightarrow \Delta(6) \oplus N \hookrightarrow
  T(6)\oplus T(4).
\]
This follows from the fact that $M(2)$ is the only contravariantly
self-dual submodule of $T(4)$ with the correct character (noting that
$L(1)$ must occur as a direct summand as it is a submodule of $T(6)$).

\subsection{}\label{844} 
By the above, the projection of $L(\sigma)\otimes L(\tau) $ onto the
 linkage class   containing the highest weight $(2p+a+b-2j,j)$, for $2j\le a$,
$a+b\le p-3$ is isomorphic to $M(2)\oplus L(1)$.

For $0\le 2c \le p-3$, consider the tensor product $L(\lfloor c/2
\rfloor, 0)\otimes L(\lceil c/2\rceil,0)$.  All weights in $C_6$ on
the ${\alpha_2}$-string through $(2p+c,0) $ are of the form
$(2p+c-2j,j)$ for $0\le 2j\le \lfloor c/2 \rfloor $, and so they all
label summands of $L(\lfloor c/2 \rfloor, 0)\otimes L(\lceil
c/2\rceil,0)$ isomorphic to $M(2)\oplus L(1)$ (the symmetric version
also holds).  Letting $c$ range over $0\le 2c \le p-3$ (respectively
$(p-3)\le 2c \le 2(p-3)$) we get that all weights in region $B$
(respectively, region $A$) of $C_6$ (respectively $C_6'$) in
Figure \ref{hoop} are of the form $(2p+c-2j,j)$ for some $0\le 2j\le
\lfloor c/2 \rfloor$.

Figure \ref{hoop} illustrates that any weight in $C_2$ is linked to
such a weight; more precisely, weights in region $A'$ are linked to
those in region $A$ and similarly for regions $B$ and $B'$.  Therefore
all $M(\lambda)$ for $\lambda \in C_2$ appear as direct summands of a
minimal tensor product.

\begin{figure}[h]
\begin{center}
\begin{tikzpicture}[scale=1]
   \path (0,0) coordinate (origin);
  \path (60:1cm) coordinate (A1);
  \path (60:2cm) coordinate (A2);
  \path (60:3cm) coordinate (A3);
   \path (120:1cm) coordinate (B1);
  \path (120:2cm) coordinate (B2);
  \path (120:3cm) coordinate (B3);
   \path (A1) ++(120:2cm) coordinate (C1);
  \path (A2) ++(120:1cm) coordinate (C2);
   \clip (-2,0) rectangle (A3);
   \foreach \i in {1,...,19}
  {
    \path (origin)++(60:0.2*\i cm)  coordinate (a\i);
    \path (origin)++(120:0.2*\i cm)  coordinate (b\i);
    \path (a\i)++(120:4cm) coordinate (ca\i);
    \path (b\i)++(60:4cm) coordinate (cb\i);
 %  \draw[thin,gray] (a\i) -- (ca\i) (a\i)--(b\i) (b\i) -- (cb\i); 
      }
   \draw[thick,gray]  (origin) -- (A3) (origin) -- (B3) (A3) -- (B3)
        (A1) -- (C1) (A2) -- (C2)  
        (B1) -- (C2) (B2) -- (C1)  
        (B1) -- (A1) (B2) -- (A2) ;
   \path (a15)   coordinate (P);
   \path (a10)++(120:0.5cm) coordinate (QB);
   \path (a10)coordinate (Q1B);
%      \path (b8)++(60:-0.4cm) coordinate (Q1B);
%   \path (a5)++(120:-1.2cm) coordinate (QB11Q);
\fill[color=cyan] (P) -- (Q1B) -- (QB)  ;
% \shadedraw[gray,opaque]  (P) -- (Q1B) -- (QB)  ;
 \path (b15)   coordinate (P1);
   \path (b10)++(60:0.5cm) coordinate (QB1);
   \path (b10)coordinate (Q1B1);
\fill[color=yellow] (P1) -- (Q1B1)  --  (QB1)    ;
% \shadedraw[gray,nearly transparent] (P1) -- (Q1B1)  --  (QB1)    ;
% \fill[pattern=crosshatch dots] (P1) -- (Q1B1)  --  (QB1)    ;
%   \path (b10)++(60:0.5cm) coordinate (QB1);
   \path (120:0.5cm) coordinate (BOBBY);   \path  (BOBBY) ++(60:0.5cm) coordinate (BOBBY1);
 \path (b5)   coordinate (redbob1);
 \path (a5)   coordinate (bluebob1);
  \path (b5)  ++(60:1cm)  coordinate (bobtop);
\fill[color=cyan] (BOBBY1)-- (redbob1) --(bobtop)  ;
\fill[color=yellow]  (BOBBY1)--(bluebob1) --(bobtop) ;
%  \shadedraw[gray,opaque]    (BOBBY1)-- (redbob1) --(bobtop)  ;
%  \shadedraw[gray,nearly transparent]   (BOBBY1)--(bluebob1) --(bobtop) ;
  %\fill[pattern=crosshatch dots]  (BOBBY1)--(bluebob1) --(bobtop) ;
  %\pattern[pattern=checkerboard,pattern color=black!30]
  %  (BOBBY1)--(bluebob1) --(bobtop) ;
%  \draw[very thick,fill,blue]   (P) -- (QB) (QB) -- (Q1B) (Q1B) -- (P)  ;%  (P)--(QBQ)     ;
 \node at (116:2.35cm) {\tiny $A$} ;
 \node at (64:2.35cm) {\tiny $B$} ;
 \node at (99:1.1cm) {\tiny $B'$} ;
 \node at (81:1.1cm) {\tiny $A'$} ;
     \end{tikzpicture}
\end{center}
\caption{Regions $A$ and $B$ contain weights of the form $(2p+c-2j,j)$
  and $(j,2p+c-2j)$, respectively, for $0\le 2j\le \lfloor c/2 \rfloor
  $.  They are linked to regions $A'$ and $B'$ respectively.}
\label{hoop}
\end{figure}
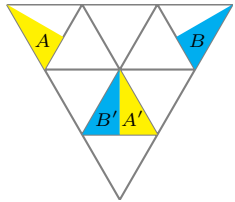
 
\subsection{}  \label{MtoM}
We have now shown that $L(\sigma) \otimes L(\tau)$ is a direct sum of
tilting modules and modules of the form $M(\nu)$, for $\nu\in C_2$.
By Section \ref{c2c2}, any tensor product of the form $L(\lambda)\otimes L(\mu)$ for
$\lambda,\mu\in C_2$ is a direct summand of $E^{\otimes r} \otimes
L(\sigma) \otimes L(\tau)$ such that $0\le r \le
2(p-3)$. 
 
%By the linkage principle, we have that $L(2)\otimes E$ is a sum of
%simple modules from the closure of $C_2$.  It is clear that for $T$ a
%tilting module, $T\otimes (E^{\otimes r}\otimes (E^\ast)^{\otimes
%s})$ is a direct sum of tilting modules.

Finally, it remains to show that $ E^{\otimes r} \otimes M(2) $ is a
direct sum of tilting modules and modules of the form $M(\nu)$ for
$\nu \in C_2$.  Any $p$-regular linkage class of $E \otimes M(\nu)$ is
of the form $M(\nu')$ for $\nu' \in C_2$, by translation.  It remains
to check that a $p$-singular direct summand of $E \otimes
M(\nu)$ is tilting.
Such a tensor product involves one or two $p$-singular linkage
classes: their highest weights are in $\mathcal{F}_{1|2}$,
$\mathcal{F}_{3|4}$ or $\mathcal{F}_{3'|4'}$.
%%%%%%%%%%%%%%%%%
%The projection onto the $p$-singular linkage class
A direct summand with highest weight in $\mathcal{F}_{1|2}$ is
immediately seen to be a simple tilting module.
%%%%%%%%%%%%%%%%%

It is easy to see that the characters of the other linkage components of  $E\otimes M(\nu)$ are
equal to the corresponding tilting characters. 
Let $\gamma=\nu+\overline{\varepsilon}_1$ be a weight in $\mathcal{F}_{3|4}$ or $\mathcal{F}_{3'|4'}$. 
 We have that
 $L(\gamma)\cong E^{[1]}\otimes L'$ where $L'$ has highest weight in alcove $C_1$. 
 Therefore $L(\gamma)	\otimes E^\ast \cong E^{[1]}\otimes (L'\otimes E^\ast)$ has no $p$-restricted composition factors.  The   head of $M(\nu)$ is $p$-restricted, therefore
 $$\Hom(M(\nu)\otimes E, L(\gamma))\cong
\Hom(M(\nu), L(\gamma)	\otimes E^\ast	)=0.$$
% A simple hom argument shows that the socle of 
%$E\otimes M(\nu)$ does not contain a submodule isomorphic
%to $L(3|4)$, therefore 
Therefore  $L(3|4)$  is not in the head of $M(\nu)\otimes E$.   
By the self-duality of $M(\nu)\otimes E$, we conclude that $L(3|4)$  is not in the socle of $M(\nu)\otimes E$.
%By self duality $M(\nu)\otimes E$ has no copy of $L(3|4)$ in the head or socle.
It follows that the linkage component of $E\otimes M(\nu)$ is the uniserial tilting module $[L(2|3'),L(3|4),L(2|3')]$ or its symmetric cousin. 

%\textcolor{red}{[The preceding paragraph needs a bit more detail. It
%    would be good to give some indication of the hom argument needed,
%    for instance.]}

\subsection{} \label{OHMYGOD}
We let $M'(2)$ denote any direct summand of the form $M(2)\oplus L(1)$
appearing in a minimal tensor product, $L(\sigma)\otimes L(\tau)$.  We
have seen in our case by case analysis, that a simple module $L(1)$
can appear as a summand of such a tensor product only if it appears as
a summand of some $M'(2)$.
 
We have seen in Section \ref{c2c2} %\textcolor{red}{[QUESTION: CORRECT REF?]} 
    that $L(\lambda)\otimes L(\mu)$, for $\lambda,\mu \in C_2$,
appears as a direct summand of $E^{\otimes r} \otimes L(\sigma)
\otimes L(\tau)$ for $0\le r \le 2(p-3)$.  By Lemma \ref{Lemma4} a simple
module of the form $L(1)$ appears in the tensor product
$L(\lambda)\otimes L(\mu)$ as a direct summand if and only if it
appears as a direct summand of some $E^{\otimes r} \otimes M'(2)$.
  
Fixing some linkage class, recall that the characters of the modules
$M'(2)$, $T(3), T(3')$ and $T(2)$ in that linkage class, are of the
form
 \begin{align*}
[M'(2)] &= [L(3)] + [L(3')] + 2[L(2)] + 2[L(1)]	\\
[T(3)] &= [L(3)] + 2[L(2)] +[L(1)]	\\
[T(3')] &=  [L(3')] + 2[L(2)] + [L(1)]	\\
[ T(2)] &= [L(2)] + 2[L(1)]	,
 \end{align*}
where each module on the right-hand side is the unique simple module in
the given linkage class.  Note that these characters are linearly
independent as the transition matrix,
\[
\left(\begin{array}{cccc}
  1 & 1 & 0 & 0 \\1 & 0 & 1 & 0 \\2 & 2 & 2
  & 1 \\2 & 1 & 1 & 2
\end{array}\right),
\] 
is non-singular.  Therefore the decomposition of a tensor product
$L(\lambda)\otimes L(\mu)$ is uniquely determined by its character.
 
\subsection{Decomposition algorithm, II}\label{algorithm:2}
It follows from the above that we can calculate the multiplicities of
the indecomposable direct summands of $L(\lambda)\otimes L(\mu)$, for
the case $\lambda, \mu \in C_2$, as follows:
 
\begin{enumerate}[label=(\alph*)]
\item Express $[L(\lambda) \otimes L(\mu)]$ in terms of the
  $[\Delta(\nu)]$-basis. This can be done using three applications of
  the Littlewood--Richardson rule as follows
  \begin{equation*} 
  [L(\lambda)\otimes L(\mu)] = \sum_{\nu \in X^+} (
  c^\nu_{\lambda,\mu} - c^\nu_{ s_{1|2} \cdot \lambda,\mu}- c^\nu_{
    \lambda,s_{1|2}\cdot\mu} +c^\nu_{ s_{1|2}\cdot
    \lambda,s_{1|2}\cdot\mu} ) [\Delta(\nu)]
\end{equation*}
 \item 
  Express $[L(\lambda) \otimes L(\mu)]$ in terms of the
  $[L(\nu)]$-basis; i.e., compute the composition factor
  multiplicities in both filtrations and their difference. 
   This produces an expression of the form 
  \[
  [L(\lambda) \otimes L(\mu)] = \textstyle \sum_\nu d^{\lambda,\mu}_\nu [L(\nu)]
  \]
  in which each $d^{\lambda,\mu}_\nu \ge 0$. 

\item If $\nu \notin C_1 \cup C_2\cup C_3 \cup C_{3'}$ is maximal such that
  $d^{\lambda,\mu}_\nu >0$, then subtract $d^{\lambda,\mu}_\nu
  [T(\nu)]$. Repeat on the difference, until there do not exist any
$\nu \notin C_1 \cup C_2\cup C_3 \cup C_{3'}$ appearing in the expression.

\item At this point, only terms of the form $[L(\nu)]$ for $\nu \in
  C_1 \cup C_2\cup C_3 \cup C_{3'}$ will remain.  Let $\nu \in C_2$ be
  a representative of a linkage class in the above and consider the
  projection onto that linkage class.  Then we are dealing with an
  expression of the form $$ b_\nu[L(\nu)] +b_{s_{1|2}
    \cdot\nu}[L({s_{1|2} \cdot\nu})]+b_{s_{2|3} \cdot\nu}[L({s_{2|3}
      \cdot\nu})]+ b_{s_{2|3'} \cdot\nu}[L({s_{2|3'} \cdot\nu})]. $$
  Therefore the multiplicities of $M'(\nu), T(s_{2|3} \cdot\nu),
  T(s_{2|3'} \cdot\nu)$ and $T(\nu)$ in the tensor product are given
  by the matrix product
  \[
     \left(\begin{array}{cccc}3/4 & 3/4 & -1/2 & 1/4 \\1/4 & -3/4 &
       1/2 & -1/2 \\ -3/4 & 1/4 & 1/2 & -1/4 \\ -1/2 & -1/2 & 0 &
       1/2\end{array}\right) \left(\begin{array}{c}b_{s_{2|3} \cdot\nu
         } \\ b_{s_{2|3'} \cdot\nu } \\ b_{ \nu } \\b_{s_{1|2}
           \cdot\nu }\end{array}\right).
 \]
 The $4 \times 4$ matrix in the above product is obtained by inverting
 the transition matrix above. The resulting multiplicities must be
 non-negative integers.
\end{enumerate}

\begin{eg} 
Suppose the characteristic is $p=5$.  Consider the tensor product
$L(3,1 )\otimes L(3,1)$.  The Littlewood--Richardson rule gives the
character
\[
  [L( 6,2 )] + 2[L( 2,4 )] + [L( 4,3 )] + [L( 7,0 )]+[L (0,5)] + 2[L
    (1,3 )] + 2[L (0,2)].
\]
The $p$-singular characters $[L( 4,3 )]$ and $[L( 6,2 )] + 2[L ( 2,4
  )] $ are both tilting.  This leaves us with a linkage class
component with character $[L( 7,0 )]+[L(0,5 )] + 2[L (1,3 )] + 2[L(0,2
  )]$.  This is the character of $M'(1,3)$.  Therefore $L(3,1 )\otimes
L(3,1) = M(1,3) \oplus T( 0,2) \oplus T(6, 2 ) \oplus T(4,3)$.
\end{eg}

% END OF COMMENTED OUT TEXT

% END OF COMMENTED OUT TEXT
%%%%%%%%%%%%%%%%%%%%%%%%%%%%%%%%%%%%%%%%%%%%%%%%%%%%%%%%%%%%%%

\end{document}